\documentclass{amsart}

\usepackage{enumerate,amssymb,amsmath}

\numberwithin{equation}{section}

\newtheorem{theorem}{Theorem}[section]

\newtheorem{proposition}[theorem]{Proposition}
\newtheorem{corollary}[theorem]{Corollary}

\newtheorem*{theorem a}{Theorem A}

\theoremstyle{definition}
\newtheorem{remark}[theorem]{Remark}
\newtheorem{definition}[theorem]{Definition} 

\newtheorem*{assumption}{Assumption~A} 
\newtheorem*{assumptionB}{Assumption~B}

\newcounter{condition}

\newcommand{\C}{\ensuremath{C}}
\newcommand{\R}{\ensuremath{\mathbb{R}}}

\newcommand{\cal}[1]{\ensuremath{\mathcal{#1}}}

\newcommand\RR\R
\newcommand\pd\partial

\newcommand{\al}{\alpha}

\newcommand{\de}{\delta}
\newcommand{\ep}{\epsilon}

\newcommand{\la}{\lambda}

\newcommand{\si}{\sigma}

\newcommand{\vp}{\varphi}

\newcommand{\De}{\Delta}

\newcommand{\La}{\Lambda}

\newcommand{\Om}{\Omega}

\def\G{\mathbf G}
\def\K{\mathbf K}

\DeclareMathOperator{\vol}{vol}

\DeclareMathOperator{\Id}{Id}

\newcommand\loc{_{\mathrm{loc}}}
\newcommand\avg{_{\mathrm{avg}}}
\newcommand\Dir{_{\mathrm{Dir}}}

\newcommand{\hf}{\widehat{f}}
\newcommand{\hu}{\widehat{u}}
\newcommand{\hg}{\widehat{g}}

\renewcommand\leq\leqslant
\renewcommand\geq\geqslant
\renewcommand\epsilon\varepsilon

\begin{document}

\title[Overdetermined boundary problems with general
nonlinearities]{Solutions to the overdetermined boundary problem for
  semilinear equations with position-dependent nonlinearities}

\author[M.~Dom\'{\i}nguez-V\'{a}zquez]{Miguel Dom\'{\i}nguez-V\'{a}zquez}
\address{Departamento de Matem\'aticas, Universidad Aut\'onoma de Madrid, and Instituto de Ciencias Matem\'aticas, CSIC-UAM-UC3M-UCM, Madrid, Spain.}

\email{miguel.dominguezv@uam.es}
\author[A.~Enciso]{Alberto Enciso}
 \address{Instituto de Ciencias Matem\'aticas, Consejo Superior de
 	Investigaciones Cient\'ificas, Madrid, Spain.}
\email{aenciso@icmat.es, dperalta@icmat.es}
\author[D.~Peralta-Salas]{Daniel Peralta-Salas}

%


\begin{abstract}
  We show that a wide range of 
  overdetermined boundary problems for semilinear equations with
  position-dependent nonlinearities admits
  nontrivial solutions. The result holds true both on~$\RR^n$ and on
  compact Riemannian manifolds. As a byproduct of the proofs we also obtain
  some rigidity, or partial symmetry, results for solutions to overdetermined problems on
  Riemannian manifolds of nonconstant curvature.
\end{abstract}


\maketitle

\section{Introduction}
Let $M$ be a compact manifold of dimension $n\geq 2$ endowed with a
Riemannian metric that we will denote by~$g$
or~$\langle\cdot,\cdot\rangle$. In this paper we are interested in
proving the existence of domains $\Omega\subset M$ where there is a nontrivial solution~$u>0$ to the overdetermined boundary value problem
\begin{equation}\label{eq:onp}
\left\{
\begin{array}{rcll}
\Delta_g u+\lambda f(\cdot,u)&=&0 & \text{in } \Omega 
\\
u&=&0 & \text{on } \partial\Omega
\\
\langle\nabla_{g}u,\nu_{g}\rangle &=& \text{constant}& \text{on } \partial\Omega.
\end{array}
\right.
\end{equation}
Here $\lambda$ is some real constant,
$f(p,z)\in\C_{\mathrm{loc}}^{1,\alpha}(M\times\R)$ is a given
nonlinearity that may depend on the point $p$ of the manifold $M$, $\Delta_g$
is the Laplacian on the manifold and, as is customary, the
Neumann condition is formulated in terms of the gradient with respect
to the metric $g$, which we denote by $\nabla_g$, and the unit outer
normal $\nu_g$ of the domain.

It should be emphasized that essentially all the existing literature
on overdetermined problems focuses on the case when the nonlinearity
only depends on~$u$, that is, $f(p,z)=G(z)$ for some function $G\in
\C^{1,\al}\loc(\RR)$, in which case the domains for which the
overdetermined problem admits a solution are usually called {\em
  $G$-extremal domains}:
\begin{equation}\label{eq:onp_laplacian}
\left\{
\begin{array}{rcll}
\Delta_g u+\lambda G(u)&=&0 & \text{in } \Omega 
\\
u&=&0 & \text{on } \partial\Omega
\\
\langle\nabla_{g}u,\nu_{g}\rangle &=& \text{constant}& \text{on } \partial\Omega.
\end{array}
\right.
\end{equation}

The investigation of overdetermined boundary value problems traces
back to Serrin's seminal paper~\cite{Serrin} in 1971, which concerns
translation-invariant overdetermined problems on~$\R^n$ of the
form~\eqref{eq:onp_laplacian}, the model case being $G(z)=1$. In this case, it is clear that there is a radial solution
to the overdetermined problem when the domain is a ball. By generalizing the
moving plane method developed by Alexandrov~\cite{Alex} in 1956 to
study constant mean curvature hypersurfaces in $\R^n$, Serrin was able
to show that, in a way, those are the only solutions to the
overdetermined problem on bounded domains. In other words, under mild assumptions it is
known that if \eqref{eq:onp_laplacian} admits a solution for a bounded domain
$\Omega$ of $\R^n$, then
$\Omega$ must be a Euclidean ball and the solution $u$ is radial. This
method works also for the hyperbolic space and
hemisphere~\cite{KP:duke} due to the existence of many totally
geodesic hypersurfaces, but fails for more 
general geometries of nonconstant curvature. Other important symmetry or classification results concerning (mainly unbounded) $G$-extremal domains in spaces of constant curvature can be found in~\cite{Brandolini,EFM:arxiv,EMao:arxiv,EMaz:arxiv,RRS:cpam,Traizet:gafa}. 

Existence or symmetry results for overdetermined problems in
Riemannian manifolds of nonconstant curvature are very
scarce. Nontrivial existence results of $G$-extremal domains in flat
spaces can be found in~\cite{PPW,RRS:jems,SS:advmath,S:cvpde}, and also for
some specific manifolds in~\cite{MS:esaim} and~\cite{EP:jmaa}. The
latter also provides a symmetry result on certain
manifolds. In fact, in general Riemannian manifolds, the very strong
requirement that constant Dirichlet and Neumann conditions be
satisfied simultaneously may even lead one to conjecture that such
solutions should not exist for generic metrics. However, in two
surprising papers, Pacard and Sicbaldi~\cite{PS:AIF} and Delay and
Sicbaldi~\cite{DS:DCDS} proved the existence of extremal domains with
small volume for the first eigenvalue $\lambda_1$ of the Laplacian in
any compact Riemannian manifold, so Problem~\eqref{eq:onp_laplacian}
admits solutions with $G(z):=z$ and $\lambda:=\lambda_1(\Om)$ for
any compact~$M$. A similar approach was followed by Fall and
Minlend~\cite{FM:ACV} to show the existence of solutions
to~\eqref{eq:onp_laplacian} when $G$ is constant.

Our first objective in this paper is to provide existence results for
the nonlinear overdetermined problem~\eqref{eq:onp} for a wide range
of nonlinearities~$f(p,z)$, possibly depending on the point $p$ in the
manifold. 
The solution domains $\Omega$ that we construct to the Problem~\eqref{eq:onp} are
perturbations of small geodesic balls centered at suitably chosen
points of the manifold. One can state the result as follows, where
$B_r$ denotes the ball in $\R^n$ of radius~$r$, $|B_r|$ is its volume,
and $f_z$, $f_{zz}$ denote
the first and second partial derivatives of $f(p,z)$ with respect to
$z$. Notice that in the hypotheses we do not impose any restrictions on the growth
of~$f(p,z)$ as $z\to\infty$ and that there are no global sign
conditions.

\begin{theorem}\label{th:main}
Let $M$ be a compact Riemannian manifold, and $f(p,z)$ any function that satisfies one of the
following two conditions:
\begin{enumerate}[{\rm (i)}]
\item $f\in C^{1,\al}\loc(M\times\RR)$ and $f(p,0)>0$ for all $p\in M$.
\item $f\in C^{2}\loc(M\times\R)$, $f(p,0)=0$,  $f_z(p,0)=c>0$ and $f_{zz}(p,0)\neq 0$ for all~$p\in M$, where $c$ is a constant independent of $p$.
\end{enumerate} 
Then for every small enough positive~$\ep$ there exists a domain $\Om\subset M$ with volume equal to $|B_\ep|$ and a positive constant $\la$, which is of order $\ep^{-2}$, such that the overdetermined problem~\eqref{eq:onp}
admits a positive solution. The domain $\Omega$ is a
$\C^{2,\alpha}$-small perturbation of a geodesic ball of
radius~$\ep$.
\end{theorem}

The construction of solutions to~\eqref{eq:onp}
builds upon ideas of Pacard--Sicbaldi~\cite{PS:AIF} and
Delay--Sicbaldi~\cite{DS:DCDS} in the case of linear,
position-independent equations, and extends them to the case of
semilinear equations with position-dependent nonlinearities. The
strategy is as
follows. Firstly, we carry out an analysis of the nonlinearities
considered in the statement of Theorem~\ref{th:main} and, using some bifurcation arguments, we prove that they satisfy certain rather
nontrivial technical conditions that are crucially employed in several
steps of the demonstration. Next we show that a domain $\Omega$ that is a small
perturbation of a small geodesic ball in $M$ admits a positive
solution $u\in \C^{2,\alpha}(\Omega)$ to $\Delta_g u+\epsilon^{-2}
f(\cdot,u)=0$ in $\Omega$, $u\rvert_{\partial\Omega}=0$, where
$\epsilon>0$ is somehow related to the size of~$\Omega$. Then we
consider an operator $\cal{F}$ that measures how far the normal
derivative of $u$ along $\partial\Omega$ is from being constant. We
calculate the linearization of $\cal{F}$, and study some properties of
the linearized operator. This allows us to apply the implicit function
theorem in Banach spaces to $\cal{F}$ and prove that, for any $p\in M$
and for any small enough $\epsilon>0$, there exists a small
perturbation $\Omega_{\epsilon,p}$ of a geodesic ball centered at $p$
that admits a solution $u_{\epsilon,p}$ satisfying not only the
equation and the Dirichlet condition, but also the Neumann condition
up to, roughly speaking, a linear error. This linear error is
controlled by certain vector field $a_{\epsilon,p}$ on~$M$. Then we
consider, for each small enough $\epsilon>0$, a smooth map $\cal{J}_\epsilon\colon M\to\R$ that associates to each point $p$ the energy of $u_{\epsilon,p}$, that is, 
\[
\cal{J}_\epsilon (p):=\int_{\Omega_{\epsilon,p}}\left(\frac{1}{2}\left|
    \nabla_g
    u_{\epsilon,p}\right|^2_g-F(\cdot,u_{\epsilon,p})\right)\, dV_g,
\]
where $F(p,z):=\int^z_0f(p,\zeta)d\zeta$ and $dV_g$ is the Riemannian 
measure induced by~$g$. The critical points of the function
$\cal{J}_\epsilon$ are precisely the zeros of the vector field~ $a_{\epsilon,p}$. Finally, it is shown that $u_{\epsilon,p}$
satisfies the Neumann condition precisely if $p$ is a critical point
of $\cal{J}_\epsilon$. If $M$ is compact, then $\cal{J}_\epsilon$ has
a critical point; indeed, it has at least as many critical points as
the Lusternik--Shnirelmann category of~$M$~\cite{LS:ams}. Therefore,
for every small enough~$\ep>0$ we obtain domains $\Om$, with
volume $|B_\ep|$ and enclosing certain specific points $p\in M$,
 where the overdetermined problem with constant of order
$\epsilon^{-2}$ admits a nontrivial solution. The points of the manifold
at which one centers the solution domains turn out to be determined by the zeros of the 
vector field $a_{\epsilon,p}$.

Notice that, as the zeros of a vector field are ultimately involved,
it is not hard to believe that when the Euler characteristic of the
manifold is nonzero the result holds in greater generality. This idea
leads to the following slightly stronger result, which allows for the
presence of first order terms. We have chosen to state it in the Introduction because, for this reason, it
applies to problems that do not have a variational structure:

\begin{theorem}\label{th:Eulerchar}
Suppose that the Euler characteristic of the compact manifold~$M$ is
nonzero and let $X$ be a vector field on~$M$ of class $C^{0,\al}$. Then the statements of Theorem~\ref{th:main}
remain true for the more general equation
\[
\Delta_g u+\langle \nabla_g u, X\rangle+\lambda f(\cdot,u)=0 \qquad  \text{in } \Omega 
\]
with the same assumptions on the nonlinearity
$f$ and the same overdetermined boundary conditions:
\[
u=0\quad\text{and}\quad  \langle\nabla_g u,\nu\rangle=\text{constant}\qquad
\text{on }\pd\Om\,.
\]
\end{theorem}

It is worth stressing that Theorem~\ref{th:main} provides new
existence results
even for overdetermined problems with the flat Laplacian on~$\RR^n$;
in particular, it automatically
ensures the existence of a positive solution for the overdetermined
problem
\begin{equation*}
\left\{
\begin{array}{rcll}
\Delta u+\lambda f(\cdot,u)&=&0 & \text{in } \Omega 
\\
u&=&0 & \text{on } \partial\Omega
\\
\partial_\nu u &=& \text{constant}& \text{on } \partial\Omega\,,
\end{array}
\right.
\end{equation*}
where $\Om\subset\R^n$ and~$\De$ denotes the standard Laplacian, for any nonlinearity~$f(x,z)$ that is
periodic in~$x$ and satisfies Condition~(i) or~(ii) in Theorem~\ref{th:main}. This is clear because one
can reformulate the problem on a flat $n$-dimensional torus, which
falls within the scope of Theorem~\ref{th:main}.

Perhaps more interestingly, one can obtain existence results on~$\R^n$ even for
nonlinearities that are non-periodic in~$x$. This is because the method of proof of Theorem~\ref{th:main} can be extended to manifolds that
are not necessarily compact but have some asymptotic symmetry
properties that one could call ``asymptotic homogeneity''. Details
are provided in Section~\ref{sec:homogeneity}. To illustrate this fact
we will next present a
very particular instance of Theorem~\ref{th:asympt_types}, which is stated and proved in Section~\ref{sec:homogeneity}:

\begin{theorem}\label{T.asympt}
Assume
that one of the following two conditions holds:
\begin{enumerate}[{\rm (i)}]
\item $f(x,z)\in C^{1,\al}\loc(\RR^n\times\RR)$ converges to $G(z)\in
  C^{1,\al}\loc(\RR)$ locally uniformly in $C^{1,\al}$ as $|x|\to\infty$ and $\displaystyle\inf_{x\in\R^n} f(x,0)>0$.
\item $f(x,z)\in C^{2}\loc(\R^n\times\R)$ converges to $G(z)\in C^2\loc(\R)$
  locally uniformly in $C^2$ as $|x|\to\infty$, $f(x,0)=0$,  $f_z(x,0)=c$ and
  $|f_{zz}(x,0)|> c'$ for all~$x\in
  \R^n$, where $c,c'$ are positive constants independent of $x$.
\end{enumerate}
Then for every small enough~$\ep$
there exists a domain $\Om\subset \RR^n$, which is a $C^{2,\al}$-small deformation of a ball of radius~$\ep$ with the same volume, and a positive constant $\la$ of order $\ep^{-2}$ such that the overdetermined problem
\begin{equation*}
\left\{
\begin{array}{rcll}
\Delta u+\lambda f( \cdot,u)&=&0 & \text{in } \Omega 
\\
u&=&0 & \text{on } \partial\Omega
\\
\partial_\nu u &=& \text{constant}& \text{on } \partial\Omega
\end{array}
\right.
\end{equation*}
admits a positive solution.
\end{theorem}

When the nonlinearity $f(p,z)=G(z)$ is independent of~$p$, the arguments involved in the proof of Theorem~\ref{th:main}
yield as a byproduct a uniqueness result that can be used to prove
the symmetry of certain $G$-extremal domains.
With the exception of~\cite{EP:jmaa}, the existing symmetry results for
overdetermined problems only apply when~$M$ is a space of constant
curvature, that is, a Euclidean space, a hyperbolic space, a round
sphere, or a quotient thereof. The intuitive reason of this is that
these symmetry results are typically obtained using the moving plane
method, so in a certain sense they hinge on the idea of using
isometries (that is, rigid motions) to transport planes. A plane is
determined by three points, and the key property of spaces of constant
curvature that is ultimately employed is that they are the only
three-point homogeneous spaces, which means that given two triples of points $(p_1,p_2,p_3)$ and $(p_1',p_2',p_3')$ on the manifold
with the same relative distances (i.e., $d(p_i,p_j)=d(p_i',p_j')$), then
there is an isometry~$\varphi$ such that $\varphi(p_j)=p_j'$.

Our symmetry results apply to Riemannian manifolds that also have a
large isometry group, but not large enough to be three-point
homogeneous, which means that the moving plane method will not work. Specifically, we prove symmetry results in harmonic
  spaces, which are defined as those Riemannian manifolds whose
geodesic spheres of sufficiently small radius have constant mean
curvature. Examples of harmonic spaces are the isotropic homogeneous
spaces (also referred to as two-point homogeneous spaces), such as the
projective and hyperbolic spaces over the distinct division algebras
(that is, the complex numbers, the quaternions and the octonions). We refer to Section~\ref{sec:symmetry} for more information on these spaces. As
we will see, the price to pay for this greater generality is that one
cannot prove the symmetry of any $G$-extremal domain, but only of those that are
close enough to being a small geodesic ball. Therefore, this is a
rigidity result for domains close to small geodesic spheres where the
overdetermined problem admits a nontrivial solution.

In order to state this kind of results, it is convenient to introduce
some notation. First, given a closed hypersurface $\Sigma$ in $M$, we define its \emph{center of mass} as the minimum of the function
\[
p\in M\mapsto\frac{1}{2}\int_{\Sigma}d^2(p,q)d\sigma_g,
\]
whenever it exists and is unique, where $d$ is the Riemannian distance
on $M$ and $d\sigma_g$ is the induced Riemannian volume form
on $\Sigma$.  With a slight abuse of notation, we will say that a
domain $\Om$ is \emph{centered at $p$} if the center of mass of its
border $\partial \Om$ is~$p$. We refer e.g.\ to~\cite{Ka:cpam,Na:agag}
for details. Second, given a point $p\in M$ and a continuous positive
function $h$ defined on the unit tangent sphere $T_p^1M$ at $p$ that
is not too large in $L^\infty$-norm (specifically, smaller than the
injectivity radius at~$p$), we find it convenient to introduce the
notation
\begin{equation}\label{ball}
B_h^g(p):=\{\exp^g_p(x):x\in T_p M, \, 0\leq | x | <h(x/| x |)\}\,,
\end{equation}
where $\exp^g$ denotes the exponential map of $(M,g)$.
It is clear that if $\Om$ is a domain that is close enough to a
geodesic ball centered at~$p$, in a sense that can be made
precise easily, then there is some function~$h$ as above such that $\Om=B_h^g(p)$.

We are now ready to state an important special case of our symmetry
results, which asserts that any $G$-extremal domain in a harmonic
space that is close enough to a small geodesic ball is indeed a
geodesic ball if the nonlinearity is, roughly speaking, concave and sublinear:

\begin{theorem}\label{th:harmonic_intro}
	Let $M$ be a harmonic space of dimension $n\geq 2$, and let
        $G\in\C_{\mathrm{loc}}^{1,\alpha}(\R)$ be concave, with
        $G(0)>0$, and such that either $G(z_0)=0$ for some~$z_0>0$ or $\lim_{z\to +\infty}\frac{G(z)}z=0$.
	 Suppose that $\Omega=B^g_{\epsilon(1+v)}(p)$ is a $G$-extremal
         domain centered at a certain point~$p$, i.e., that there is a
         nontrivial solution to the overdetermined
         problem~\eqref{eq:onp_laplacian} in this domain.
  If $\ep$, $\la\ep^2$
  and~$\|v\|_{\C^{2,\alpha}(T_p^1M)}$ are small enough, then
  $\Omega$ is a geodesic ball and $u$ only depends on the
  geodesic distance to~$p$.
\end{theorem}

This result applies, for instance, to the model overdetermined
problem $G(z)=1$ on a harmonic space, showing that any domain close to a small
enough geodesic ball where the corresponding overdetermined problem
admits a solution is exactly a geodesic ball. The same symmetry result
turns out to hold for the extremal domains associated with the first
eigenvalue problem, that is, Problem~\eqref{eq:onp_laplacian} with $\la\, G(z)=\la_1(\Om)z$ (see Remark~\ref{R.la1}).

This paper is organized as follows. In Section~\ref{S.ass} we state
a technical hypothesis that we call Assumption~A and give sufficient conditions ensuring
that it is satisfied, which correspond to conditions~(i)--(ii)
in Theorem~\ref{th:main}. In Section~\ref{sec:setup} we prove
Theorem~\ref{th:main} using that the nonlinearities we consider
satisfy Assumption~A. This involves solving the Dirichlet problem for small
perturbations of small geodesic balls and studying the operator that
encodes the failure of the Neumann condition. Then, we use a geometric variational approach
to show that the solutions to~\eqref{eq:onp} correspond to the
critical points of the functions $\cal{J}_\epsilon$, which are
analyzed subsequently. In Section~\ref{S.variations} we consider
variations on the proof of Theorem~\ref{th:main}, proving, in
particular, Theorem~\ref{th:Eulerchar}. In Section~\ref{sec:homogeneity} we discuss
the possibility of relaxing the compactness assumption on~$M$ and show
the validity of our existence results under different approximate homogeneity
hypotheses, establishing Theorem~\ref{T.asympt} as a particular case. Finally, Section~\ref{sec:symmetry} is devoted to 
several symmetry results for $G$-extremal domains in specific
Riemannian manifolds, such as two-point homogeneous, symmetric or
harmonic spaces, including Theorem~\ref{th:harmonic_intro}.

\section{Analysis of the nonlinearities considered in Theorem~\ref{th:main}}
\label{S.ass}

In this section we show that the nonlinearities considered in
Theorem~\ref{th:main} satisfy certain technical properties that will
be extensively used in the paper and which we call
Assumption~A. To state the assumptions, we will henceforth use the notations
$\nabla$ and $\Delta$ for the Euclidean gradient and Laplacian,
respectively. Throughout we are assuming that the function $f$ is
of class $C^{1,\al}\loc(M\times\RR)$.

Roughly speaking, Assumption~A asserts that, when the point on the manifold is
``frozen'', the nonlinearity is such that there are positive solutions
to the associated Dirichlet problem in the unit ball
of~$\RR^n$, and that a linear partial differential operator and countably
many linear ODEs defined in terms of this function satisfy certain
nondegeneracy conditions. Specifically, Assumption~A can be stated as follows: 

\begin{assumption}	
For each point $p\in M$ there exists a radial
        function $\phi_p(x)$ of class $C^{2,\al}$ on the unit ball,
        varying differentiably with~$p$, and a constant $\bar{\lambda}>0$
        independent of~$p$ such
        that:
\begin{enumerate}
\item $\phi_p(x)$ is positive in~$B_1$ and solves the Dirichlet problem
\begin{equation*}
		\left\{
		\begin{array}{rcll}
		\Delta_x \phi_p(x)+\bar{\lambda}\, f(p,\phi_p(x))&=&0 & \text{in } B_1\,,
		\\
		\phi_p&=&0 & \text{on } \partial B_1,
		\end{array}
		\right.
		\end{equation*}
with $\partial_\nu \phi_p\rvert_{\partial B_1}<0$.

\item The linear operator $\Delta+\bar{\lambda} f_z(p,\phi_p)\colon \C^{2,\alpha}_{\mathrm{Dir}}(B_1)\to \C^{0,\al}(B_1)$ is invertible, where $\C^{2,\alpha}_{\mathrm{Dir}}(B_1)$ denotes the subspace of functions of $C^{2,\al}(B_1)$ that vanish on $\partial B_1$.

\item For each nonnegative integer~$j$, let $a_{p,j}(r)$ be the only solution to the equation
\begin{multline}\label{eq:a_j}
			a_{p,j}''(r)+\frac{n-1}{r}a_{p,j}'(r)+\bigg(\bar{\lambda} \,
                        f_z(p,\phi_p(r))-\frac{j(j+n-2)}{r^2}\bigg)\, a_{p,j}(r)\\
=\frac{n-1-j(j+n-2)}{r^2}\partial_r\phi_p(r)
			\end{multline}
                        with initial conditions $a_{p,j}(1)=a_{p,j}'(1)=0$.
                        Here, since $\phi_p(x)$ is a radial function, we are
                        writing $\phi_p(r)$ with the obvious
                        meaning. Then
\[
\lim_{r\to0^+} a_{p,j}(r)\neq0
\]
for all $j\geq2$.
\end{enumerate}
\end{assumption}

\begin{remark}\label{R.assA}
  The characteristic exponents of the ODE~\eqref{eq:a_j} at~0 are $-j$
  and $n+j-2$, so the meaning of condition~(iii) is that, for all
  $j\geq2$ and $p\in M$, $r=1$ is not a double zero of the solution to
  the only solution to the ODE~\eqref{eq:a_j} that is continuous
  at~$r=0$.
\end{remark}

\begin{remark}\label{rem:bar_lambda}
	Using the constant~$\la$ that appears in the
	overdetermined problem~\eqref{eq:onp},  without any loss of generality we take the constant $\bar{\lambda}$ appearing in Assumption A to be
	$\bar{\lambda}=1$, after multiplying the function~$f$ by a constant if
	necessary. We will do this without any further mention when we use
	Assumption~A in subsequent sections.
\end{remark}

In the following theorem we provide a wide class of nonlinearities
that satisfy Assumption~A, corresponding to the first class of
examples provided in the Introduction:

\begin{theorem}\label{T.ex1}
Suppose that $M$ is compact and $f(p,0)>0$ for all $p\in M$. Then Assumption~A holds for all small enough $\bar{\lambda}>0$.
\end{theorem}

\begin{proof}
We start with the first assertion of Assumption~A. For this, consider the
map
\[
\cal{G}\colon \RR\times C^{2,\al}\Dir(B_1)\to C^\al(B_1)
\]
given by
\[
\cal{G}(\bar{\lambda},\phi):= \De\phi+\bar{\lambda} \, f(p,\phi)\,,
\]
where $p$ is any fixed point on~$M$. Since $\cal{G}(0,0)=0$ and the partial derivative
\[
D_\phi\cal{G}(0,0)=\De
\]
is invertible with inverse $C^\al(B_1)\to C^{2,\al}\Dir(B_1)$, it
follows from the implicit function theorem that for all~$\bar{\lambda}$ in a small interval centered at~0 there is
a unique function~$\phi_{p,\bar{\lambda}}$ in a small neighborhood of~$0$ in
$C^{2,\al}\Dir(B_1)$ that satisfies the equation
$\cal{G}(\bar{\lambda}, \phi_{p,\bar{\lambda}})=0$. Furthermore it depends smoothly on $p\in M$ and $\bar{\lambda}\in \R$, and it is bounded as
\begin{equation}\label{boundphi}
\|\phi_{p,\bar{\lambda}}\|_{C^{2,\al}(B_1)}\leq C|\bar{\lambda}|\,.
\end{equation}

Take now a small enough positive $\La_1\in\RR$. In view of the
bound~\eqref{boundphi} one can now assume that
\begin{equation}\label{boundphi2}
\|\phi_{p,\bar{\lambda}}\|_{L^\infty}<1
\end{equation}
for all $\bar{\lambda}<\La_1$, so the corresponding function
$\phi_{p,\bar{\lambda}}$ then satisfies
\begin{align*}
-\De \phi_{p,\bar{\lambda}}&=\bar{\lambda}\, f(p, \phi_{p,\bar{\lambda}})\\
&\geq \bar{\lambda}\, \big( f(p,0) -
  \|f_z(p,\cdot)\|_{L^\infty((0,1))}\,
  \|\phi_{p,\bar{\lambda}}\|_{L^\infty}\big)\\
&\geq \bar{\lambda} \,\big( f(p,0)-C'\,\bar{\lambda}\big) >0
\end{align*}
for some constant $C'>0$, provided that $\bar{\lambda}$ is smaller than some positive constant~$\La_1'$. Since $\phi_{p,\bar{\lambda}}|_{\pd B_1}=0$, the maximum
principle then ensures that~$\phi_{p,\bar{\lambda}}>0$ in~$B_1$, and Hopf's
boundary point lemma shows that $\pd_\nu\phi_{p,\bar{\lambda}}<0$.

Let us now pass to the second point of Assumption~A. Let us
take~$\bar{\lambda}<\La_2$, where $\La_2$ is chosen so
that $0<\La_2\leq \La_1'$ and also
\[
\La_2\,\|f_z\|_{L^\infty(M\times[0,1])}<\la_1(B_1)\,,
\]
where $\la_1(B_1)$ is the first Dirichlet eigenvalue of the unit
ball. With a positive $\bar{\lambda}<\La_2$ we automatically obtain that $\Delta+\bar{\lambda}\, f_z(p,\phi_p)$ is
coercive, i.e., that there is a positive constant~$c$ such that
\begin{equation}\label{coercive}
\int_{B_1}\left(| \nabla u(x)|^2-\bar{\lambda}\,
  f_z(p,\phi_{p,\bar{\lambda}}(x))\,u(x)^2\right)\, dx\geq c\|u\|_{H^1(B_1)}^2
\end{equation}
for all $u\in H^1_0(B_1)$. It is standard that the inverse of
$\Delta+\bar{\lambda} \, f_z(p,\phi_{p,\bar{\lambda}})$ with Dirichlet boundary condition then exists
as a map $\C^{\alpha}(B_1)\to \C^{2,\alpha}(B_1)$.

We now pass to the third point of Assumption~A. We are going to prove that the
coercitivity condition~\eqref{coercive} implies the third condition. Equivalently, for each integer $j\geq 2$, let $b_{p,j}(r)$ satisfy
        Equation~\eqref{eq:a_j} with the boundary conditions $b_{p,j}'(0)=0$ and $b_{p,j}(1)=0$. We have to show that $b_{p,j}'(1)\neq 0$ for each $j\geq 2$.
			
			First we claim that $b_{p,j}\leq 0$ for all $j\geq
                        2$. To prove it, let us apply the inequality~\eqref{coercive} to a radial
                        function $u\in H^1_0(B_1)$, which we simply
                        denote by $u(r)$, to obtain
			\begin{equation}\label{eq:negative_radial}
			\int_0^1\left(u'(r)^2-\bar{\lambda} f_z(p,\phi_{p,\bar{\lambda}}(r))u(r)^2\right)r^{n-1}\,
                        dr\geq 0.
			\end{equation}
			If $b_{p,j}\geq 0$ in an interval $[r_1,r_2]$ with $b_{p,j}(r_i)=0$, then multiplying \eqref{eq:a_j} by $b_{p,j}r^{n-1}$ and integrating by parts yields
			\[
			\int_{r_1}^{r_2}\left((b_{p,j}')^2-\bar{\lambda} f_z(p,\phi_{p,\bar{\lambda}})b_{p,j}^2+\frac{1}{r^2}j(j+n-2)b_{p,j}^2\right)r^{n-1}\,
                        dr\leq 0,
			\]
			which, by \eqref{eq:negative_radial}, implies that $b_{p,j}\equiv 0$ on $[r_1,r_2]$. If $b_{p,j}\geq 0$ in an interval $[0,r_2]$ with $b_{p,j}(r_2)=0$, a similar argument (using $b_{p,j}'(0)=0$) shows that $b_{p,j}\equiv 0$ on $[0,r_2]$. Altogether, this proves the claim.
			
			Suppose now that $b_{p,j}'(1)=0$, for some $j\geq 2$. Then, since $b_{p,j}(1)=0$ and $b_{p,j}\leq 0$ for all $j\geq 2$, we have $b_{p,j}''(1)\leq 0$, but evaluating \eqref{eq:a_j} at $r=1$ yields
			\begin{align*}
			0&=(n-1)b_{p,j}'(1)=(n-1-j(j+n-2))\phi_{p,\bar{\lambda}}'(1)-b_{p,j}''(1)
			\\
			&\geq (n-1-j(j+n-2))\phi_{p,\bar{\lambda}}'(1)>0,
			\end{align*}
			where we have used that
\[
\phi_{p,\bar{\lambda}}'(1)=\pd_\nu \phi_{p,\bar{\lambda}}|_{\pd B_1}<0
\]
by
                     Hopf's boundary point lemma. This gives a
                        contradiction. Therefore, $b_{p,j}'(1)>0$ for
                        all $j\geq 2$. This completes the proof of
                        the theorem upon noticing that all the smallness
                        conditions that we impose on~$\bar{\lambda}$ are
                        uniform in $p\in M$ by the compactness of~$M$. 
\end{proof}

Now, our aim in the rest of this section is to study the second class
of nonlinearities considered in Theorem~\ref{th:main}, namely those functions
$f(p,z)$ that vanish at $z=0$ and satisfy certain additional assumptions. In order to show that these nonlinearities satisfy Assumption~A (which is done in Theorem~\ref{T.ex2} below), we need to introduce some preliminary definitions and results. These will also play an important role later in Section~\ref{sec:setup}.

Let us assume that a nonlinearity $f\in C^{1,\al}_{\mathrm{loc}}(M\times\R)$ satisfies conditions (i) and (ii) in Assumption~A, with $\bar{\lambda}=1$. That is, for each point $p\in M$, there exists a positive radial
function $\phi_p(x)$ of class $C^{2,\al}$ on the unit Euclidean ball,
varying differentiably with~$p$ and solving the Dirichlet problem
	\begin{equation*}
	\left\{
	\begin{array}{rcll}
	\Delta_x \phi_p(x)+f(p,\phi_p(x))&=&0 & \text{in } B_1\,,
	\\
	\phi_p&=&0 & \text{on } \partial B_1,
	\end{array}
	\right.
	\end{equation*}
for each $p\in M$. Moreover, for each $p\in M$, the linear operator 
\[
L_p:=\Delta+ f_z(p,\phi_p)\colon \C^{2,\alpha}_{\mathrm{Dir}}(B_1)\to\C^{0,\alpha}(B_1)
\] 
is invertible.

Let $\C^{k,\alpha}\avg(S^{n-1})$ denote the space of functions on $S^{n-1}$ of class $C^{k,\alpha}$ whose integral on $S^{n-1}$ vanishes. By the invertibility of $L_p$ we can consider, for each $w\in\C^{2,\alpha}\avg(S^{n-1})$, the unique solution $\psi_w$ in $\C^{2,\alpha}(B_1)$~to
\begin{equation}\label{eq:psi_dirichlet}
\left\{
\begin{array}{rcll}
L_p\psi_w&=&0 & \text{in } B_1
\\
\psi_w&=&-c_1 w& \text{on } \partial B_1,
\end{array}
\right.
\end{equation}
where $c_1:=\partial_r\phi_p\rvert_{r=1}$ is the outward radial derivative of $\phi_p$ on $\partial B_1$. Since $\phi_p$ is radial, $c_1$ is a constant. The assignment $w\mapsto\psi_w$ allows us to define the operator
\begin{equation}\label{eq:H}
\begin{array}{rccl}
H_p\colon & \C^{2,\alpha}\avg(S^{n-1}) & \to & \C^{1,\alpha}\avg (S^{n-1})
\\
&w& \mapsto & (\partial_r\psi_w+c_2 w)\rvert_{\partial B_1},
\end{array}
\end{equation}
where $c_2:=\partial^2_r\phi_p\rvert_{r=1}$ is another
constant. Indeed, we will see that $H_p$ takes values in
$\C^{1,\alpha}\avg(S^{n-1})$.

Denote by $\Delta_{S^{n-1}}$ the Laplace-Beltrami operator of the unit
sphere $S^{n-1}$ of $\R^n$. It is well known that the eigenvalues of
$-\Delta_{S^{n-1}}$ are $\mu_j:=j(n+j-2)$ with $j$ ranging over the
nonnegative integers. Denote by~$V_j$ the corresponding eigenspaces,
which are given by the restriction to the unit sphere of the space of
harmonic polynomials on~$\R^n$ that are homogeneous of degree~$j$. Then we have:

\begin{proposition}\label{prop:H}
	The operator $H_p$ satisfies the following properties:
	\begin{enumerate}[{\rm (a)}]
		\item $H_p\colon\C^{2,\alpha}\avg(S^{n-1}) \to \C^{1,\alpha}\avg(S^{n-1})$ is a self-adjoint, first order elliptic operator that preserves the eigenspaces of $-\Delta_{S^{n-1}}$.
		\item It can be defined alternatively by 
		\[
		H_p(w)=\partial_r \Psi_w\rvert_{\partial B_1}, 
		\]
		where, for each $w\in C^{2,\alpha}\avg(S^{n-1})$, $\Psi_w$ is the unique continuous solution to
		\begin{equation}\label{eq:alternative_H}
		\left\{
		\begin{array}{rcll}
		L_p\,\Psi_w&=&\frac{1}{r^2}\partial_r\phi_p(\Delta_{S^{n-1}}+n-1)w & \text{in } B_1
		\\
		\Psi_w&=&0& \text{on } \partial B_1.
		\end{array}
		\right.
		\end{equation}
		\item The kernel of $H_p$ contains the eigenspace $V_1$ of $-\Delta_{S^{n-1}}$ with eigenvalue $n-1$.
		\item The following three statements are equivalent: 
		\begin{enumerate}[{\rm (1)}]
			\item $\ker H_p=V_1$,
			\item $f$ satisfies condition (iii) of Assumption A with $\bar{\lambda}=1$,
			\item there exists a constant $c_p>0$ such that $\|H_p(w)\|_{C^{1,\al}(S^{n-1})}\geq c_p\|w\|_{C^{2,\al}(S^{n-1})}$ for all $w\in C^{2,\al}(S^{n-1})\cap (V_0\oplus V_1)^\perp$.
		\end{enumerate}
	\end{enumerate}
\end{proposition}
\begin{proof}
	The operator $H_p$ is a linear combination of the
	Dirichlet-to-Neumann operator for $L_p=\Delta+f_z(p,\phi_p)$ and the identity. Hence, $H_p$ is elliptic
	of first order, essentially self-adjoint on the space of
	smooth functions, and lower bounded, meaning that $H_p\geq -
        C$ in the sense that
\[
\int_{S^{n-1}} w\, H_p(w)\, d\si\geq- C \|w\|_{L^2}^2
\]
for all $w$ in the Sobolev space $H^1(S^{n-1})$.
	
	In order to prove the other statements, we will first give an alternative description of the operator $H_p$. Thus, for each $w\in C^{2,\alpha}\avg(S^{n-1})$ we consider the unique continuous solution $\Psi_w$ to \eqref{eq:alternative_H}; since $\partial_r\phi_p(0)=0$, standard estimates
	guarantee the existence of such $\Psi_w$ and its regularity
        away from the origin. Using the
	formula
	\[
	\Delta h=\partial_r^2 h+\frac{n-1}{r}\partial_r
	h+\frac{1}{r^2}\Delta_{S^{n-1}}h
	\]
	and the fact that 
	\[
	\Delta \partial_r \phi_p=-f_z(p,\phi_p)\partial_r\phi_p
	+\frac{n-1}{r^2}\partial_r\phi_p\,,
	\]
	some direct calculations show that $\Psi_w$ is given by 
\[
\Psi_w(y)=\psi_w(y)+\partial_r\phi_p(y) w(y/|y|)
\]
for $y\in B_1$, where $\psi_w$ is the solution to \eqref{eq:psi_dirichlet}. Hence, we have that
	\[
	H_p(w)=\partial_r \Psi_w\rvert_{\partial B_1}, \quad \text{for each } w\in C^{2,\alpha}\avg(S^{n-1}).
	\]
	
	This alternative definition easily implies that $H_p$ preserves the eigenspaces $V_j$ of $-\Delta_{S^{n-1}}$, its image lies in the space of functions with average zero in $S^{n-1}$, and $V_1\subset \ker H_p$. 
	
	In order to conclude the proof we have to show claim (d). Thus, let
	\[
	w=\sum_{j=1}^\infty  w_j\,,
	\]
	with $w_j\in V_j$, be the
	eigenfunction decomposition of
	$w\in C^{2,\alpha}\avg(S^{n-1})$. Then an easy computation
        shows that
	\[
	H_p(w)=\sum_{j=1}^\infty \alpha_{p,j} w_j\,,
	\]
	where
	$\alpha_{p,j}=b_{p,j}'(1)$ and $b_{p,j}$ is the continuous solution
	to~\eqref{eq:a_j} with $b_{p,j}(1)=0$.
	Since $V_1\subset \ker H_p$, we have $\alpha_1=0$. Condition (iii) of Assumption~A is equivalent to $\alpha_{p,j}=b_{p,j}'(1)\neq 0$ for each $j\geq 2$, which happens precisely when $V_1=\ker H_p$. Since $H_p$ is elliptic and lower bounded, its eigenvalues satisfy $\lim_{j\to\infty}\alpha_{p,j}=+\infty$ and, hence, $\ker H_p=V_1$ if and only if there exists a positive constant $c_p$ such that $\|H_p(w)\|_{C^{1,\al}(S^{n-1})}\geq c_p\|w\|_{C^{2,\al}(S^{n-1})}$ for all $w\in C^{2,\al}_{\mathrm{avg}}(S^{n-1})\cap V_1^\perp$. This concludes the proof.
\end{proof}

Now we are in a position to show that the second kind of nonlinearities mentioned in Theorem~\ref{th:main} satisfy Assumption~A:

\begin{theorem}\label{T.ex2}
Suppose that $M$ is compact, the function $f$ is of class $C^2\loc(M\times\R)$ and, for all $p\in M$, it satisfies $f(p,0)=0$,  $f_z(p,0)=c>0$, where $c$ is independent of $p$, and $f_{zz}(p,0)\neq 0$.
Then Assumption~A holds for all $\bar{\lambda}$ in an open interval one of whose endpoints is $\lambda_1(B_1)/c$, where $\lambda_1(B_1)$ is the first Dirichlet eigenvalue of the Laplacian of the Euclidean unit ball.
\end{theorem}

\begin{proof}
As in the proof of Theorem~\ref{T.ex1}, for each fixed $p$, consider the operator $\cal{G}\colon \RR\times C^{2,\al}\Dir(B_1)\to C^\al(B_1)$ given by
\[
\cal{G}(\la,\phi):= \De\phi+\la \, f(p,\phi)\,.
\]
By assumption, we have that $\cal{G}(\lambda,0)=0$, for all $\lambda\in\RR$.
Define 
\[
\lambda_{0}:=\lambda_1(B_1)/f_z(p,0)=\lambda_1(B_1)/c,
\]
Then the partial derivative
\[
D_\phi\cal{G}(\lambda_{0},0)=\Delta+\lambda_1(B_1)
\]
is an operator from $C^{2,\al}\Dir(B_1)$ to $C^\alpha (B_1)$ whose
kernel is generated by the first eigenfunction $\varphi_1$ of the
Dirichlet Laplacian of $B_1$ (which is normalized to have unit
$L^2$-norm), and whose range is the subspace
$\langle\varphi_1\rangle^\perp$ of functions in $C^{\alpha}(B_1)$ that
are $L^2$-orthogonal to $\varphi_1$, so in particular it has
codimension one in $C^\alpha(B_1)$. We have used here the well-known fact that the first
eigenvalue has multiplicity one.
Moreover, by assumption we have that the second partial derivative
\[
D_{\lambda\phi}\cal{G}(\lambda_{0},0)\varphi_1=c\,\varphi_1\neq 0
\]
does not belong to the range $\langle\varphi_1\rangle^\perp$ of $D_\phi\cal{G}(\lambda_0,0)=\Delta+\lambda_1(B_1)$.
Then we are in the conditions of Crandall-Rabinowitz Theorem~\cite[Lemma~1.1]{CR:arma}, whence we obtain the existence of $\lambda_{p,s}\in\R$ and $\phi_{p,s}\in\C^{2,\alpha}\Dir(B_1)$, both continuously differentiable depending on $s\in (-s_0,s_0)$ and on $p\in M$, such that $\lambda_{p,0}=\lambda_0$, $\phi_{p,0}=0$, and $\cal{G}(\lambda_{p,s},\phi_{p,s})=0$, that is,
\begin{equation*}
\Delta \phi_{p,s}+\lambda_{p,s}f(p,\phi_{p,s})=0.
\end{equation*}
Moreover, the same result guarantees that $\phi_{p,s}=s(\varphi_1+\widetilde{\phi}_{p,s})$ where $\widetilde{\phi}_{p,s}\in\langle\varphi_1\rangle^\perp$ is a $C^1$-curve with $\widetilde{\phi}_{p,0}=0$. Thus, for all $s\in(0,s_0)$ (possibly taking a smaller $s_0$), the function $\phi_{p,s}$ is positive in $B_1$, and also $f(p,\phi_{p,s})>0$ in $B_1$. Then, by Hopf's boundary point lemma, $\partial_\nu \phi_{p,s}\rvert_{B_1}<0$. Hence, for each fixed $s\in(0,s_0)$, $\phi_{p,s}$ satisfies item~(i) in Assumption~A for $\bar{\lambda}=\lambda_{p,s}$.

Now, \cite[Corollary~1.13]{CR:arma} guarantees the existence of $\varphi_{p,\lambda}\in\C^{2,\alpha}\Dir(B_1)$ and $\gamma_{p}(\lambda)\in\R$, both depending continuously differentiably on $\lambda$ in an open interval around $\lambda_0$, and of $v_{p,s}\in\C^{2,\alpha}\Dir(B_1)$ and $\mu_{p,s}\in\R$,  both depending continuously differentiably on $s\in(-s_0,s_0)$, such that
\begin{equation}\label{eq:varphi_gamma}
(\Delta+\lambda f_z(p,0))\varphi_{p,\lambda}=\gamma_{p}(\lambda)\varphi_{p,\lambda}
\end{equation}
and
\begin{equation}\label{eq:L_p,s}
L_{p,s}v_{p,s}=\mu_{p,s}v_{p,s},
\end{equation}
where $L_{p,s}$ is the linear operator given by
\[
L_{p,s}:=\Delta+\lambda_{p,s}f_z(p,\phi_{p,s}).
\]
Moreover,
\[
\gamma_{p}(\lambda_0)=0=\mu_{p,0},\qquad \varphi_{p,\lambda_0}=v_{p,0}=\varphi_1, \qquad v_{p,s}-\varphi_1\in \langle\varphi_1\rangle^\perp\cap C^{2,\alpha}\Dir(B_1),
\]
and $\mu_{p,s}$ is a simple eigenvalue of $L_{p,s}$. Now, if the nonlinearity satisfies $f_{zz}(p,0)<0$ for all $p\in M$, then $\lambda_{p,s}>\lambda_0$ for all $s>0$ small enough (see~\cite[Corollary~1.1]{Lions}). If $f_{zz}(p,0)>0$ for all $p\in M$, then $\lambda_{p,s}<\lambda_0$ for all $s>0$ small enough (see~\cite[Remark~1.4]{Lions}). In any case, the derivative $\partial_s\lambda_{p,s}$ cannot be identically zero in any interval around $s=0$. Since \cite[Theorem~1.16]{CR:arma} guarantees that $\mu_{p,s}$ and $-s\,\partial_s\lambda_{p,s}$ have the same zeroes near $s=0$, it follows that $\mu_{p,s}\neq 0$ for infinitely many $s$ near $s=0$. Then, the same result states that 
\[
\lim_{\substack{s\to 0\\\mu_{p,s}\neq 0}} \frac{-s\, \partial_s \lambda_{p,s}\,\gamma'(\lambda_0)}{\mu_{p,s}}=1,
\]
which, by L'H\^opital's rule, implies
\begin{equation}\label{eq:from_limit}
-(\partial_s\lambda_{p,0})\gamma'(\lambda_0)=\partial_s\mu_{p,0},
\end{equation}
where $\partial_s\lambda_{p,0}$ and $\partial_s\mu_{p,0}$ denote the partial derivatives of $\lambda_{p,s}$ and $\mu_{p,s}$ with respect to $s$ at $s=0$, respectively.

Relation~\eqref{eq:varphi_gamma} can be rewritten as
$(\Delta+a(\lambda))\varphi_{p,\lambda}=0$, where $a(\lambda)=\lambda
f_z(p,0)-\gamma_{p}(\lambda)\in\R$. Then, $a(\lambda)$ is an
eigenvalue of $-\Delta$ with Dirichlet boundary conditions. Since
$a(\lambda_0)=\lambda_1(B_1)$, by the discreteness of the eigenvalues
of $-\Delta$ and the continuity of $a(\lambda)$ with respect to
$\lambda$, we deduce that that $a(\lambda)=\lambda_1(B_1)$ for
all~$\la$ close to~$\la_0$, that is
\[
\gamma_{p}(\lambda)=\lambda f_z(p,0)-\lambda_1(B_1),
\]
for all $\lambda$ in a neighborhood of $\lambda_0$. Hence, 
\begin{equation}\label{eq:gamma'}
\gamma'_{p}(\lambda_0)=f_z(p,0)=c>0. 
\end{equation}
Similarly, since $\mu_{p,0}=0$ is the smallest eigenvalue of $-L_{p,0}=-(\Delta+\lambda_1(B_1))$, we have that $\mu_{p,s}$ is also the smallest eigenvalue of $-L_{p,s}$, for all $s\in(-s_0,s_0)$. Hence, taking a smaller $s_0$ if necessary, for each $s\in(-s_0,s_0)$, $L_{p,s}\colon C^{2,\alpha}_{\mathrm{Dir}}(B_1)\to C^{0,\alpha}(B_1)$ is invertible if and only if $\mu_{p,s}\neq 0$. Let us show that this indeed holds for all small enough $s>0$.

By differentiating~\eqref{eq:L_p,s} with respect to $s$ we get
\begin{align*}
0={}&\Delta \partial_s v_{p,s}+\partial_s\lambda_{p,s} f_z(p,\phi_{p,s})v_{p,s}+\lambda_{p,s}f_{zz}(p,\phi_{p,s})\partial_s\phi_{p,s}v_{p,s}
\\
&+\lambda_{p,s}f_z(p,\phi_{p,s})\partial_s v_{p,s}-\partial_s\mu_{p,s} v_{p,s}-\mu_{p,s}\partial_s v_{p,s},
\end{align*}
which, evaluated at $s=0$, multiplied by $\varphi_1$ and integrating in $B_1$, yields
\begin{equation*}
0=f_z(p,0)\partial_s\lambda_{p,0}+\lambda_0f_{zz}(p,0)\int_{B_1}\varphi_1^3-\partial_s\mu_{p,0}.
\end{equation*}
This, together with~\eqref{eq:from_limit}, \eqref{eq:gamma'} and the assumption $f_{zz}(p,0)\neq 0$, implies that
\begin{equation*}
\partial_s\lambda_{p,0}=-\frac{1}{2c}\lambda_0f_{zz}(p,0)\int_{B_1}\varphi_1^3\neq 0\,,
\end{equation*}
where we are using that $\vp_1>0$. Indeed, $\lambda_{p,s}$ is strictly increasing in $s$ if $f_{zz}(p,0)<0$, whereas $\lambda_{p,s}$ is strictly decreasing in $s$ if $f_{zz}(p,0)>0$, for $s\in(-s_0,s_0)$, maybe taking a smaller $s_0$. Taking into account that $\mu_{p,s}$ and $-s\,\partial_s\lambda_{p,s}$ have the same zeroes, we deduce that $\mu_{p,s}\neq 0$ for all $s\in(s_0,s_0)\setminus\{0\}$. Thus, $L_{p,s}\colon C^{2,\alpha}_{\mathrm{Dir}}(B_1)\to C^{0,\alpha}(B_1)$ is invertible, which proves item (ii) in Assumption~A.

Fix $p\in M$ and $s\in(-s_0,s_0)\setminus\{0\}$. Similarly as in the discussion before Proposition~\ref{prop:H}, for each $w\in C_{\mathrm{avg}}^{2,\alpha}(S^{n-1})$ consider the unique solution $\psi_{p,s,w}\in C^{2,\alpha}(B_1)$ to
\begin{equation}\label{eq:psi_p,s,w}
\left\{
\begin{array}{rcll}
L_{p,s}\psi_{p,s,w}&=&0 & \text{in } B_1
\\
\psi_{p,s,w}&=&-c_1 w& \text{on } \partial B_1,
\end{array}
\right.
\end{equation}
where $c_1:=\partial_r\phi_{p,s}\rvert_{r=1}\in\R$. 

Now, for each $p\in M$ and $s\in(-s_0,s_0)\setminus\{0\}$, associated to $L_{p,s}$ we consider the operator $H_{p,s}$ defined analogously as in~\eqref{eq:H}. By Proposition~\ref{prop:H}, $H_{p,s}$ can be defined alternatively by
\[
H_{p,s}(w)=\partial_r \Psi_{p,s,w}\rvert_{\partial B_1}, 
\]
where, for each $w\in C^{2,\alpha}\avg(S^{n-1})$ and $s\neq 0$, $\Psi_{p,s,w}$ is the unique continuous solution to
\begin{equation}\label{eq:Psi_p,s}
\left\{
\begin{array}{rcll}
L_{p,s}\Psi_{p,s,w}&=&\frac{1}{r^2}\partial_r\phi_{p,s}(\Delta_{S^{n-1}}+n-1)w & \text{in } B_1
\\
\Psi_{p,s,w}&=&0& \text{on } \partial B_1.
\end{array}
\right.
\end{equation}
For $s\neq 0$ we can write 
\[
\Psi_{p,s,w}=s\, \eta_{p,s,w},
\]
for some $\eta_{p,s,w}\in \C^{0}_{\mathrm{Dir}}(B_1)$. Using this, the identity $\phi_{p,s}=s(\varphi_1+\widetilde{\phi}_{p,s})$, and dividing by~$s$, \eqref{eq:Psi_p,s} implies that $\eta_{p,s,w}$ is the only continuous solution to
\begin{equation}\label{eq:L_p,s-bis}
\left\{
\begin{array}{rcll}
L_{p,s}\eta_{p,s,w}&=&\frac{1}{r^2}\partial_r(\varphi_1+\widetilde{\phi}_{p,s})(\Delta_{S^{n-1}}+n-1)w & \text{in } B_1
\\
\eta_{p,s,w}&=&0& \text{on } \partial B_1,
\end{array}
\right.
\end{equation}
for each $s\neq 0$ and $w\in C^{2,\al}_{\mathrm{avg}}(S^{n-1})$. 

Since $v_{p,s}$ is a positive function in $B_1$, a result by Gidas, Ni and Nirenberg~\cite[Theorem~1]{GNN} guarantees that $v_{p,s}$ is radial, and hence, $\partial_r v_{p,s}\rvert_{r=1}$ is constant on $S^{n-1}$. Then, by the second equation in~\eqref{eq:psi_p,s,w} and since $\int_{S^{n-1}}w=0$, we have
\begin{equation}\label{eq:v_p,s-radial}
\int_{S^{n-1}} \psi_{p,s,w}\partial_r v_{p,s}=-\int_{S^{n-1}}c_1 w \,\partial_r v_{p,s}=0.
\end{equation}
Now, by taking the product of the first equation in~\eqref{eq:psi_p,s,w} with $v_{p,s}\in C^{2,\al}\Dir(B_1)$, integrating by parts and using \eqref{eq:L_p,s} and~\eqref{eq:v_p,s-radial}, we have
\begin{align*}
0&=\int_{B_1}v_{p,s}L_{p,s}\psi_{p,s,w}
\\
&=\int_{B_1}\psi_{p,s,w} \Delta v_{p,s}-\int_{S^{n-1}}\psi_{p,s,w} \partial_r v_{p,s}+\int_{B_1}\lambda_{p,s}f_z(p,\phi_{p,s})\psi_{p,s,w} v_{p,s}
\\
&=\int_{B_1}\psi_{p,s,w}\left(\mu_{p,s}-\lambda_{p,s}f_z(p,\phi_{p,s})\right)v_{p,s}+\int_{B_1}\lambda_{p,s}f_z(p,\phi_{p,s})\psi_{p,s,w} v_{p,s}
\\
&=\mu_{p,s}\int_{B_1}v_{p,s}\psi_{p,s,w}.
\end{align*}
Thus, if $\mu_{p,s}\neq 0$, then $\psi_{p,s,w}$ is $L^2$-orthogonal to
$v_{p,s}$. It follows that $\eta_{p,s,w}$ is also $L^2$-orthogonal to
$v_{p,s}$, and therefore to the kernel of~$L_{p,s}$, for each $s\neq 0$.

Since $L_{p,0}=\Delta+\lambda_1(B_1)$ and $\widetilde{\phi}_{p,0}=0$,
one can now take the limit $s\to 0$ in~\eqref{eq:L_p,s-bis} to deduce
that
\[
\eta_{p,0,w}:=\lim_{s\to 0} \eta_{p,s,w}
\]
is the unique continuous solution in $\langle\varphi_1\rangle^\perp$ to
\begin{equation}\label{eq:Psi_PS}
\left\{
\begin{array}{rcll}
\Delta\eta_{p,0,w}+\lambda_1(B_1)\eta_{p,0,w}&=&\frac{1}{r^2}\partial_r\varphi_1(\Delta_{S^{n-1}}+n-1)w & \text{in } B_1
\\
\eta_{p,0,w}&=&0& \text{on } \partial B_1.
\end{array}
\right.
\end{equation}
Now, for each $s\in(-s_0,s_0)$ we define
\[
\widetilde{H}_{p,s}(w):=\partial_r\eta_{p,s,w}\rvert_{\partial B_1}.
\]
Hence $H_{p,s}(w)=s\,\widetilde{H}_{p,s}(w)$, for all $s\in(-s_0,s_0)$ and $w\in C^{2,\al}_{\mathrm{avg}}(S^{n-1})$.

The boundary problem~\eqref{eq:Psi_PS} is precisely the one in~\cite[(4.7)]{PS:AIF}, and the operator $\tilde{H}_{p,0}$ is the one in~\cite[(4.8)]{PS:AIF}. It was proved in~\cite{PS:AIF} that $\|\tilde{H}_{p,0}(w)\|_{C^{1,\al}(S^{n-1})}\geq C \|w\|_{C^{2,\al}(S^{n-1})}$ for all $w$ that is $L^2$-orthogonal to $V_0\oplus V_1$, for some constant $C\in \R$. By the continuity of $H_{p,s}$ with respect to $s$ we have that $\|\tilde{H}_{p,s}(w)\|_{C^{1,\al}(S^{n-1})}\geq \frac{C}{2} \|w\|_{C^{2,\al}(S^{n-1})}$, $s\in(-s_0,s_0)$, maybe taking a smaller $s_0$. Hence
\[
\|H_{p,s}(w)\|_{C^{1,\al}(S^{n-1})}=s\,\|\tilde{H}_{p,s}(w)\|_{C^{1,\al}(S^{n-1})}\geq \frac{C\, s}{2} \|w\|_{C^{2,\al}(S^{n-1})},
\]
for all $s\in(-s_0,s_0)$ and $w\in C^{2,\al}_{\mathrm{avg}}(S^{n-1})\cap V_1^\perp$. Therefore, by Proposition~\ref{prop:H}(d) the nonlinearity $f$ satisfies condition (iii) of Assumption A with $\bar{\lambda}=\lambda_{p,s}$ and $\phi_p=\phi_{p,s}$, for each $s\in(-s_0,s_0)\setminus\{0\}$. This concludes the proof of
the theorem upon observing that all the smallness
conditions that we impose on~$s$ (and hence on $|\lambda_{p,s}-\lambda_0|$) are
uniform in $p\in M$ by the compactness of~$M$.
\end{proof}

\section{Proof of Theorem~\ref{th:main}}\label{sec:setup}

Our objective in this section is to show that any nonlinearity for
which Assumption~A holds satisfies the conclusions of
Theorem~\ref{th:main}:

\begin{theorem}\label{th:main2}
Let $M$ be a compact Riemannian manifold, and $f$ any function in
$C^{1,\al}\loc(M\times\RR)$ that satisfies Assumption~A. Then for every small enough positive~$\ep$ there exists a domain $\Om\subset M$ with volume equal to $|B_\ep|$ and a positive constant $\la$, which is of order $\ep^{-2}$, such that the overdetermined problem~\eqref{eq:onp}
admits a positive solution. The domain $\Omega$ is
$\C^{2,\alpha}$-small perturbation of a geodesic ball of
radius~$\ep$.
\end{theorem}

In particular, this implies
Theorem~\ref{th:main} by virtue of Theorems~\ref{T.ex1}
and~\ref{T.ex2}. 

The proof consists of three steps. As these steps are quite involved,
we will include several auxiliary
propositions. Just as in the rest of the proof, in
these propositions we assume without further mention that Assumption~A
holds with $\bar{\lambda}=1$ (see~Remark~\ref{rem:bar_lambda}). Many
of the arguments we will use in this section are modeled on ideas due
to Pacard--Sicbaldi~\cite{PS:AIF} and Delay--Sicbaldi~\cite{DS:DCDS}.

\subsection*{Step 1: Analysis of the rescaled equation with
  Dirichlet boundary conditions}

Using the notation~\eqref{ball}, the solution domains that we will
construct can be described as
$\Om=B^g_{\epsilon(1+v)}(p)$, where $v\colon T_p^1M\to(-1,\infty)$ is
a function on the unit sphere of the tangent space at~$p$ of class
$\C^{2,\alpha}(T_p^1M)$, $\epsilon>0$, and $p\in M$. We will decompose
$v$ as a sum $v=v_0+\bar{v}$ of a constant $v_0$ and a function
$\bar{v}\in \C^{2,\alpha}\avg(T_p^1M)$. But, for the sake of convenience, instead of considering a domain that depends on the function $v=v_0+\bar{v}$, we will work on a fixed domain, namely on the unit ball $B_1$ of $\R^n$, endowed with a metric that depends on $\epsilon$ and $v$. 

In order to do that, we introduce some notation. Fixed a point $p\in
M$, let $T_\epsilon$ be the homothety of $T_p M$ (which is
diffeomorphic to $\R^n$) given by $T_\epsilon(y)=\epsilon y$, $y\in
T_pM$. Given $v\in\C^{2,\alpha}(T_p^1M)$, we consider the open subset 
\[
B_{1+v}:=\{x\in T_pM: |x|<1+v(x/|x|)\}
\]
of $T_p M$ and the parametrization $\beta\colon B_1\to B_{1+v}$ given by
\begin{equation}\label{eq:beta}
\beta(y):=\biggl(1+v_0+\chi(y)\bar{v}\biggl(\frac{y}{|y|}\biggr)\biggr)y,
\end{equation}
where $\chi$ is a fixed radial cutoff function that vanishes for $|y|\leq 1/2$ and is equal to $1$ for $|y|\geq 3/4$. Then, we define $Y\colon B_1\to B^g_{\epsilon(1+v)}$ as the composition 
\begin{equation}\label{eq:Y}
Y:=\exp^g_p\,\circ \,T_\epsilon\circ \beta.
\end{equation} 
Of course, $\beta$ and $Y$ depend on $p$, $\epsilon$ and $v$,  but we
remove this dependence from the notation for the sake of
simplicity. Note that $Y(B_1)=\{p\}$ if $\epsilon=0$, whereas if
$\epsilon>0$ and  $\epsilon\|1+v\|_{L^\infty}$ is less than the
injectivity radius~$R$ of the manifold at~$p$, then $Y$ is a diffeomorphism. 

Let $x=(x^1,\dots,x^n)$ be normal coordinates in $B_R^g(p)$ around
$p$, and let $g_{ij}$ be the components of the metric $g$ in these
coordinates. Then, for each $\epsilon\in[0,R/2]$, we consider the
metric $\bar{g}$ whose entries in the coordinates $x=(x^1,\dots,x^n)$
are 
\[
\bar{g}_{ij}(x):=g_{ij}(\epsilon x)\,.
\]
Note that this metric is smooth in the neighborhood $B_2$ of $B_1$ in
$T_pM$, where $B_r$ is the ball of radius $r$ in $T_pM$ with respect
to the metric~$g_p$ at the point~$p$. Equivalently,
$\bar{g}$ can be defined in a coordinate-independent fashion as
\[
\bar{g}=\epsilon^{-2}(\exp_p^g\,\circ
\,T_\epsilon)^*(g\rvert_{B^g_R(p)})\,.
\]
Hence, $\bar{g}$ is a metric (on an open subset of $T_p M$) that is
homothetic to the pullback metric $(\exp_p^g\,\circ\,
T_\epsilon)^*g$. Moreover, in the limit $\ep\to0$ the metric $\bar{g}$
tends to the Euclidean metric $g_0$, so we will define $\bar{g}|_{\ep=0}$
this way.

For each $p$, $\epsilon$ and $v$ such that
$\epsilon\|1+v\|_{L^\infty}<R$, we define the metric 
\[
\hg:=\beta^*\bar{g}=\epsilon^{-2}Y^*g
\]
on the ball $B_1\subset T_pM$. Given a nonlinearity
$f\in\C^{1,\alpha}_{\mathrm{loc}}(M\times\R)$, we consider the
function $\hf\in \C^{1,\alpha}_{\mathrm{loc}}(B_1\times \R)$ such that
\[
\hf(y,z)=f(Y(y),z)
\]
for $y\in B_1$ and $z\in\R$. In other words, $\hf(y,z)$ is just the
expression of the function $f(p,z)$ in the local coordinates~$y$.
 
Here and henceforth, by taking coordinate charts in $M$, we identify $T_pM$ with $\R^n$, the unit tangent sphere $T_p^1M$ of $T_pM$ with the unit sphere $S^{n-1}$ of $\R^n$, and the unit ball of $T_pM$ (with respect to the inner product $g_p$) with $B_1\subset\R^n$.

\begin{proposition}\label{prop:implicit}
There exists $\delta_0>0$ such that for each positive $\de<\de_0$,
$p\in M$, $\epsilon\in[0,\de)$, and
$\bar{v}\in\C^{2,\alpha}\avg(T_p^1M)$ with
$\|\bar{v}\|_{\C^{2,\alpha}(T_p^1M)}<\delta$, there exist a unique
constant $v_0=v_0(p,\epsilon,\bar{v})\in (-C\de,C\de)$ and a unique positive function $\hu=\hu(p,\epsilon,\bar{v})\in\C^{2,\alpha}(B_1)$ with $\|\hu-\phi_p\|_{\C^{2,\alpha}(B_1)}<C\de$, and such that
\begin{equation}\label{eq:hat_vol}
\vol_{\hg}(B_1)=|B_1|,
\end{equation}
 with $v:=v_0+\bar{v}$, and 
\begin{equation}\label{eq:hat_dirichlet}
\left\{
\begin{array}{rcll}
\Delta_{\hg} \hu+\hf(\cdot,\hu)&=&0 & \text{in } B_1
\\
\hu&=&0 & \text{on } \partial B_1.
\end{array}
\right.
\end{equation}
Here $C$ is a uniform constant. Moreover, $\hu$ and $v_0$ depend smoothly on $p$, $\epsilon$ and $\bar{v}$. If $\epsilon=0$ and $\bar{v}=0$, then $u=\phi_p$ and $v_0=0$.
\end{proposition}
\begin{proof}
Observe that if $\epsilon=0$ and $\bar{v}=0$, then $\hg=(1+v_0)^2 g_0$, $dV_{\hg}=(1+v_0)^{n}dV$, where $g_0$ is the Euclidean metric, $dV$ is the Euclidean volume form, and $\hf(y,z)=f(p,z)$, for $y\in B_1$, $z\in\R$; in particular, we have that $\Delta_{\hg}\hu=(1+v_0)^{-2}\Delta \hu$. Thus, if $\epsilon=0$ and $\bar{v}=0$, by Assumption~A a solution to~\eqref{eq:hat_dirichlet} satisfying \eqref{eq:hat_vol} is given by $\hu=\phi_p$ and $v_0=0$.  

Let us recall that $\C^{2,\alpha}\Dir(B_1)$ is the subspace of functions in $\C^{2,\alpha}(B_1)$ that vanish on the boundary of $B_1$. Let $\cal{U}$ be an open neighborhood of $(0,0,0)$ in $[0,\infty)\times\C^{2,\alpha}\avg(S^{n-1})\times\R$ such that $\epsilon(1+\bar{v}+v_0)$ is less than the injectivity radius of $(M,g)$, for all $(\epsilon,\bar{v},v_0)\in\cal{U}$. Then the map 
\[
\begin{array}{rccl}
\cal{N}\colon & M\times \cal{U}\times \C^{2,\alpha}\Dir(B_1) & \to & \C^{0,\alpha}(B_1)\times\R
\\
&(p,\epsilon,\bar{v},v_0,\psi)& \mapsto & (\Delta_{\hg}\psi+\hf(\cdot,\psi),\vol_{\hg}(B_1)-|B_1|),
\end{array}
\]
where $|\cdot|$ denotes the Euclidean volume, is smooth. Since $\hg=(1+v_0)^2 g_0$ and $\Delta_{\hg}\psi=(1+v_0)^{-2}\Delta \psi$ for $\epsilon=0$ and $\bar{v}=0$, we have
\begin{align*}
\cal{N}(p,0,0,0,\psi)&=(\Delta\psi+f(p,\psi),0),
\\
\cal{N}(p,0,0,v_0,\phi_p)&=\bigl((1+v_0)^{-2}\Delta\phi_p+f(p,\phi_p),((1+v_0)^n-1)|B_1|\bigr),
\end{align*}
for every $p\in M$, $\psi\in\C^{2,\alpha}\Dir(B_1)$ and $v_0\in\R$ close enough to zero. Thus, for all $p\in M$, we have
\begin{align*}
D_\psi\cal{N}(p,0,0,0,\phi_p)&=(\Delta +f_z(p,\phi_p),0),
\\
D_{v_0}\cal{N}(p,0,0,0,\phi_p)&=(2f(p,\phi_p),n|B_1|).
\end{align*}
Therefore, since by Assumption~A the operator
\[
\Delta +f_z(p,\phi_p)\colon\C^{2,\alpha}\Dir(B_1)\to\C^{0,\alpha}(B_1)
\]
is invertible, the partial differential
\[
D_{(v_0,\psi)}\cal{N}(p,0,0,0,\phi_p)\colon\R\times
\C^{2,\alpha}\Dir(B_1)\to\C^{0,\alpha}(B_1)\times\R
\]
is also invertible, for every $p\in M$. Thus, for any fixed $p\in M$,
the implicit function theorem guarantees that, for each
$(q,\epsilon,\bar{v})$ in a neighborhood of $(p,0,0)$ in
$M\times[0,\infty)\times\C^{2,\alpha}\avg(S^{n-1})$, there exists a
unique $(v_0,\hu)$ in a neighborhood of $(0,\phi_p)$ in $\R\times
\C^{2,\alpha}\Dir(B_1)$ such that $\cal{N}(q,\epsilon,
\bar{v},v_0,\hu)=(0,0)$. Moreover, since $\phi_p>0$ and $\partial_\nu \phi_p\rvert_{\partial B_1}<0$, maybe taking
a smaller neighborhood we have that $\hu>0$. Hence, if $\de$ is a
small enough positive constant, which, by the compactness of~$M$, means
that $\de$ be smaller than some positive constant~$\de_0$ independent
of the point~$p\in M$, then for each positive constant $\ep<\de$ and each $\bar{v}$ with $\C^{2,\alpha}$-norm
smaller than $\delta$,  there exist unique $v_0$ in a neighborhood of~$0$
and $\hu\in\C^{2,\alpha}\Dir(B_1)$ in a neighborhood of $\phi_p$ satisfying
\eqref{eq:hat_vol} and \eqref{eq:hat_dirichlet}. Furthermore,
\[
|v_0|+\|\hu-\phi_p\|_{\C^{2,\alpha}(B_1)}<C\de
\]
with some uniform constant~$C$.
This completes the
proof.
\end{proof}

\subsection*{Step 2: Analyzing when the normal derivative is constant}

Let $\omega_n:=|S^{n-1}|$ be the area of the $(n-1)$-dimensional unit sphere of~$\R^n$. We consider the operator
\begin{equation}\label{eq:F}
\cal{F}(p,\epsilon,\bar{v}):=\hg(\nabla_{\hg}\hu,\nu_{\hg})\rvert_{\partial B_1}-\frac{1}{\omega_n}\int_{\partial B_1}\hg(\nabla_{\hg}\hu,\nu_{\hg})d\sigma,
\end{equation}
where $\nu_{\hg}$ is the outward unit normal vector field to $\partial
B_1$ with respect to the metric~$\hg$, $(v_0,\hu)$ is the solution
provided by Proposition~\ref{prop:implicit}, and $d\sigma$ is
the standard area measure on the unit sphere. As follows from the last
part of the proof of Proposition~\ref{prop:implicit}, $\cal{F}$ is
well-defined from a neighborhood of $M\times\{0\}\times\{0\}$ in
$M\times[0,\infty)\times\C_{m}^{2,\alpha}(S^{n-1})$ into
$\C\avg^{1,\alpha}(S^{n-1})$. Of course,
$\cal{F}(p,\epsilon,\bar{v})=0$ exactly when the normal derivative
$\hg(\nabla_{\hg}\hu,\nu_{\hg})\rvert_{\partial B_1}$ is constant.

Recall from the discussion after the proof of Theorem~\ref{T.ex1} in Section~\ref{S.ass} that for each $w\in\C^{2,\alpha}\avg(S^{n-1})$ we considered the unique (by Assumption~A) solution $\psi_w$~to
\begin{equation}\label{eq:psi_dirichlet2}
\left\{
\begin{array}{rcll}
\Delta\psi_w+f_z(p,\phi_p)\psi_w&=&0 & \text{in } B_1
\\
\psi_w&=&-c_1 w& \text{on } \partial B_1,
\end{array}
\right.
\end{equation}
where $c_1:=\partial_r\phi_p\rvert_{r=1}$ is a constant. The assignment $w\mapsto\psi_w$ allowed us to define the operator
\[
\begin{array}{rccl}
H_p\colon & \C^{2,\alpha}\avg(S^{n-1}) & \to & \C^{1,\alpha}\avg(S^{n-1})
\\
&w& \mapsto & (\partial_r\psi_w+c_2 w)\rvert_{\partial B_1},
\end{array}
\]
where $c_2:=\partial^2_r\phi_p\rvert_{r=1}$ is another
constant.

The reason for which we are interested in the operator $H_p$ is that it is the linearization of the operator $\cal{F}$ defined in \eqref{eq:F}, as we show below:

\begin{proposition}\label{prop:linearizationF}
The differential of $\cal{F}$ with respect to $\bar{v}$ at $\epsilon=0$ and $\bar{v}=0$ is $H_p$, i.e.\ $D_{\bar{v}} \cal{F}(p,0,0)=H_p$, for all $p\in M$. 
\end{proposition}
\begin{proof}
Fix $p\in M$. We have to show that $\partial_s\rvert_{s=0} \cal{F}(p,0,sw)=H_p(w)$, for every $w\in \C^{2,\alpha}\avg(S^{n-1})$. 
For the moment, let us assume that $f(p,\cdot)$ is a real analytic
function of one variable.

We set $\bar{v}=s w$. For $\epsilon=0$ we have that $\bar{g}=g_0$ and $\hg=\beta^*g_0$. Then the solution $\hu$ given in Proposition~\ref{prop:implicit} satisfies
\[
\left\{
\begin{array}{rcll}
\Delta_{\hg}\hu+f(p,\hu)&=&0 & \text{in } B_1
\\
\hu&=&0 & \text{on } \partial B_1,
\end{array}
\right.
\]
and $\vol_{\hg} (B_1)=|B_1|$. Since we are assuming that $f(p,\cdot)$ is analytic, the Cauchy--Kowalevski theorem guarantees that the solution~$\phi_p$ can be extended to an open set containing the closed unit ball of $\R^n$, and it satisfies $\Delta \phi_p+f(p,\phi_p)=0$ therein. Hence, $\widehat{\phi}_p=\phi_p\circ \beta$ is a solution to $\Delta_{\hg}\widehat{\phi}_p+f(p,\widehat{\phi}_p)=0$ in $B_1$.

By setting $\widehat{\psi}=\hu-\widehat{\phi}_p$ we observe that
\begin{equation}\label{eq:psi_hat}
\left\{
\begin{array}{rcll}
\Delta_{\hg}\widehat{\psi}-f(p,\widehat{\phi}_p)+f(p,\widehat{\psi}+\widehat{\phi}_p)&=&0 & \text{in } B_1
\\
\widehat{\psi}&=&-\widehat{\phi}_p& \text{on } \partial B_1.
\end{array}
\right.
\end{equation}
By construction, $\widehat{\psi}$ and $v_0$ depend smoothly on $s$. Moreover, if $s=0$, then $\hg=g_0$, $v_0=0$, $\widehat{\phi}_p=\phi_p$ and $\widehat{\psi}=0$. Let $\dot{\psi}=\partial_s\widehat{\psi}\rvert_{s=0}$ and $\dot{v}_0=\partial_s v_0\rvert_{s=0}$. Taking derivatives in~\eqref{eq:psi_hat} with respect to $s$ and evaluating at $s=0$, we get
\begin{equation}\label{eq:psi_dot}
\left\{
\begin{array}{rcll}
\Delta \dot{\psi}+f_z(p,\phi_p)\dot{\psi}&=&0 & \text{in } B_1
\\
\dot{\psi}&=&-c_1(\dot{v}_0+w)& \text{on } \partial B_1,
\end{array}
\right.
\end{equation}
where $c_1=\partial_r\phi_p\rvert_{r=1}$. Now, by differentiating the relation $\vol_{\hg} (B_1)=|B_1|$ with respect to $s$ at $s=0$, we get
\begin{equation}\label{eq:deriv_v0}
0=\partial_s\rvert_{s=0}\vol_{\hg}(B_1)=\partial_s\rvert_{s=0}|B_{1+v}|=\int_{S^{n-1}}g_0(\Xi,\partial_r)d\sigma=\int_{S^{n-1}}(\dot{v}_0+w)d\sigma,
\end{equation}
where we have used the differentiation formula for moving regions~\cite[Appendix~C.4]{Evans}, and where $\Xi=(\dot{v}_0+w)\partial_r$ is the velocity field of the moving boundary $\partial B_{1+v}$ at $s=0$. Since $\int_{S^{n-1}}w\,d\sigma=0$, we deduce that $\dot{v}_0=0$.

Altogether, we have shown that $v_0=\cal{O}(s^2)$ and $\hu=\widehat{\phi}_p+s\psi_w+\cal{O}(s^2)$, where $\psi_w$ is the solution to~\eqref{eq:psi_dirichlet2}. Moreover, recalling from~\eqref{eq:beta} the definition of $\beta$, in $B_1\setminus B_{3/4}$ we can write
\begin{align}\label{eq:hat_u_expansion}
\hu(y)&=\phi_p\left(((1+sw(y/|y|))y\right)+s\psi_w(y)+\cal{O}(s^2)
\\ \nonumber
&=\phi_p(y)+s\left(w(y/|y|)r\partial_r\phi_p+\psi_w\right)+\cal{O}(s^2),
\end{align}
where $r=|y|$. Now, some calculations show that, in polar coordinates $y=rz$, $r>0$, $z\in S^{n-1}$, the metric $\hg=\beta^*g_0$ in $B_1\setminus B_{3/4}$ adopts the expression 
\[
\hg=(1+v_0+sw)^2dr^2+2s(1+v_0+sw)r\,dw\,dr+r^2(1+v_0+sw)^2 g_0\rvert_{S^{n-1}}+s^2r^2dw^2.
\]
Using this and the fact that $v_0=\cal{O}(s^2)$, one gets that the outward unit normal vector field to $\partial B_1$ with respect to the metric $\hg$ is given by
\[
\nu_{\hg}=(1-sw+\cal{O}(s^2))\partial_r+\cal{O}(s)\partial_{z_j},
\]
where $\partial_{z_j}$ are coordinate vector fields on $S^{n-1}$. This, together with~\eqref{eq:hat_u_expansion}, yields 
\[
\hg(\nabla_{\hg}\hu,\nu_{\hg})=\nu_{\hg}(\hu)=c_1+s(c_2w+\partial_r\psi_w)+\cal{O}(s^2)\qquad \text{on } \partial B_1,
\]
where $c_2=\partial_r^2\phi_p\rvert_{r=1}$. This concludes the proof in the case that $f(p,\cdot)$ is analytic.

If $f\in\C^{1,\alpha}\loc(M\times \R)$, let
$f_\delta\in\C^{\omega}(\R)$, $\delta\in(0,\delta_0)$, be a family of
analytic functions converging to $f_0:=f(p,\cdot)$ locally in the $\C^{1,\alpha}$-norm. Consider the smooth map
\[
\begin{array}{rccl}
\cal{N}_p\colon & [0,\delta_0)\times \C^{2,\alpha}\Dir(B_1) & \to & \C^{0,\alpha}(B_1)
\\
& (\delta,\psi) & \mapsto & \Delta \psi +f_\delta\circ\psi.
\end{array}
\]
Then $\cal{N}_p(0,\phi_p)=\Delta \phi_p+f(p,\phi_p)=0$, and
moreover 
$$
D_\psi\cal{N}_p(0,\phi_p)=\Delta+f_0'\circ\phi_p=\Delta+f_z(p,\phi_p)
$$
is invertible as an operator from $\C^{2,\alpha}_{\mathrm{Dir}}(B_1)$
to $\C^{0,\alpha}(B_1)$ by Assumption~A. Hence, the implicit function
theorem guarantees the existence of
$\phi_{p}^{\delta}\in\C^{2,\alpha}\Dir(B_1)$ smoothly depending on
$\delta\in[0,\delta_0)$, for a maybe smaller $\delta_0>0$, such that
$\Delta \phi_{p}^{\delta}+f_\delta\circ\phi_{p}^{\delta}=0$ in $B_1$,
and $\Delta +f_\delta'\circ\phi_{p}^{\delta}$ is
invertible. Therefore, the arguments above apply to each one of the
analytic nonlinearities $f_\delta$, $\delta\in(0,\delta_0)$, with
associated solutions $\phi_{p}^{\delta}\in\C^{2,\alpha}\Dir(B_1)$ to
$\Delta \phi_{p}^{\delta}+f_\delta\circ\phi_{p}^{\delta}=0$, and
associated operators $\cal{F}^\delta$ and $H_{p}^\delta$. Thus,
$D_{\bar{v}}\cal{F}^\delta(p,0,0)=H_{p}^{\delta}$ for each
$\delta\in(0,\delta_0)$. Since both $\cal{F}^\delta$ and $H_p^\delta$
depend continuously on $\delta$, we deduce that 
$$
D_{\bar{v}}\cal{F}(p,0,0)=\lim_{\delta\to
  0}D_{\bar{v}}\cal{F}^\delta(p,0,0)=\lim_{\delta\to
  0}H_p^\delta=H_p\,,
$$
which concludes the proof.
\end{proof}

Let us now recall that the first eigenspace of $-\De_{S^{n-1}}$,
$V_1$, is the restriction to the unit sphere of the space of linear
functions on $T_pM$, which one can write as
\[
\langle a,\cdot\rangle
\]
with $a\in T_pM$. Using this fact together with
Propositions~\ref{prop:linearizationF} and~\ref{prop:H}, combined with
some ideas about the center of mass, one can prove the following result via the implicit function theorem:

\begin{proposition} \label{prop:affine}
There exists $\epsilon_0>0$ such that, for each
$\epsilon\in[0,\epsilon_0)$ and each $p\in M$, there is a unique
$\bar{v}_{\epsilon,p}\in\C^{2,\alpha}\avg(T_p^1M)$ and a unique
$a_{\epsilon,p}\in T_pM$, both smoothly depending on $p$ and
$\epsilon$, such that
$\|\bar{v}_{\epsilon,p}\|_{\C^{2,\alpha}(T_p^1M)}<\epsilon_0$, 
\[
\cal{F}(p,\epsilon,\bar{v}_{\epsilon,p})+\langle
a_{\epsilon,p},\cdot\rangle=0 \qquad \text{ on }\partial B_1=S^{n-1},
\]
and the center of mass of $\partial B^g_{\epsilon(1+v_0+\bar{v})}(p)$
is~$p$. Moreover, $\bar{v}_{0,p}=0$, $a_{0,p}=0$, and the map
\[
p\in M\mapsto a_{\epsilon,p} \in T_pM
\]
defines a $C^{1,\alpha}$-vector field $a_\ep$ on $M$.
\end{proposition}

\begin{proof}
	In~\cite[Lemma~2.7]{Na:agag} it was proved that there exists a smooth map $A$ such that $(\exp^{g}_p)^{-1}(c(p,\epsilon,v))=\epsilon A(p,\epsilon,v)$, where $c(p,\epsilon,v)$ denotes the center of mass of $\partial B^g_{\epsilon(1+v)}(p)$ as defined in the Introduction, and
	\begin{equation}\label{eq:A}
	A(p,0,v)=\frac{\int_{S^{n-1}}(1+v(z))^{n-1}z\sqrt{|\nabla v(z)|^2+(1+v(z))^2}\,dz}{\int_{S^{n-1}}(1+v(z))^{n-2}\sqrt{|\nabla v(z)|^2+(1+v(z))^2}\,dz}.
	\end{equation}
		
	Define the operator $Q=\Id -\Pi_1$, where $\Pi_1$ is the orthogonal projection of $L^2(S^{n-1})$ onto the first eigenspace $V_1$ of the spherical Laplacian $\Delta_{S^{n-1}}$. We consider the smooth map 
	\[
	\bar{\cal{F}}(p,\epsilon, \bar{v})=(A(p,\epsilon,\bar{v}+v_0(p,\epsilon,\bar{v})),Q \cal{F}(p,\epsilon,\bar{v}))
	\]
	defined from a neighborhood of $M\times\{(0,0)\}$ in $M\times [0,\infty)\times \C^{2,\alpha}\avg(S^{n-1})$ into a neighborhood of $(0,0)$ in $\R^n\times(\C^{1,\alpha}\avg(S^{n-1})\cap V_1^\perp)$. Using~\eqref{eq:A}, by considering the Taylor expansions up to order one of the numerator and denominator in $A(p,0,t\bar{v}+v_0(p,0,t\bar{v}))$ (cf.~\cite[Lemma~2.9]{Na:agag}) and taking into account that $v_0(p,0,t\bar{v})=\cal{O}(t^2)$ (as follows from~\eqref{eq:deriv_v0}), one easily shows that 
	\begin{equation}\label{eq:dA}
	\left.\frac{d}{dt}\right\vert_{t=0}A(p,0,t\bar{v}+v_0(p,0,t\bar{v}))=\frac{n}{\omega_{n}}\int_{S^{n-1}}\bar{v}(z)zdz.
	\end{equation}
	Hence, the differential of the second component of
        $\bar{\cal{F}}$ with respect to $\bar{v}$ at the point
        $(p,0,0)$ has image $\R^n$ and kernel contained in
        $V_1^\perp$. Moreover, we have that
        $\bar{\cal{F}}(p,0,0)=(0,0)$ and, according to
        Proposition~\ref{prop:linearizationF}, the differential of the
        second component of $\bar{\cal{F}}$ with respect to $\bar{v}$
        at $(p,0,0)$ is $Q  H_p$, whose kernel is $V_1$, and whose
        image is $\C^{1,\alpha}\avg(S^{n-1})\cap V_1^\perp$, by
        Proposition~\ref{prop:H}. Altogether we deduce that the
        differential $D_{\bar{v}}\bar{\cal{F}}(p,0,0)$ is
        invertible. Hence, the implicit function theorem and the
        compactness of $M$ guarantee the existence, for each $p\in M$
        and for all~$\epsilon$ smaller than certain~$\epsilon_0$, of a
        unique $\bar{v}_{\epsilon,p}$, smoothly depending on $p$ and
        $\epsilon$, with
        $\|\bar{v}_{\epsilon,p}\|_{\C^{2,\alpha}\avg(T_p^1M)}<\epsilon_0$,
        such that
        $\bar{\cal{F}}(p,\epsilon,\bar{v}_{\epsilon,p})=(0,0)$. Taking
        the unique $a_{\epsilon,p}\in T_pM$ such that
        $\Pi_1(\cal{F}(p,\epsilon,\bar{v}_{\epsilon,p}))+\langle
        a_{\epsilon,p},\cdot\rangle=0$, where
        $\langle\cdot,\cdot\rangle$ is the inner product on $T_pM$
        determined by $g$, the result follows.
\end{proof}


\subsection*{Step 3: Existence of zeros of the vector field
  $a_{\epsilon}$ and conclusion of the proof}

In the previous section we have constructed, for each $\epsilon\in(0,\epsilon_0)$ and for each $p\in M$, a function $v_{\epsilon,p}=v_0(\epsilon,p,\bar{v}_{\epsilon,p})+\bar{v}_{\epsilon,p}\in\C^{2,\alpha}(T_pM)$, a vector $a_{\epsilon,p}\in T_pM$ and a function $\hu_{\epsilon,p}\in\C^{2,\alpha}\Dir(B_1)$ such that the center of mass of $\partial B^g_{\epsilon(1+v_{\epsilon,p})}(p)$ is~$p$,  $\vol_{\hg}(B_1)=|B_1|$, and
\begin{equation}\label{eq:onp_affineB1}
\left\{
\begin{array}{rcll}
\Delta_{\hg} \hu_{\epsilon,p}+\hf(\cdot,\hu_{\epsilon,p})&=&0 & \text{in } B_1
\\
\hu_{\epsilon,p}&>&0 & \text{in } B_1
\\
\hu_{\epsilon,p}&=&0 & \text{on } \partial B_1
\\
\hg(\nabla_{\hg}\hu_{\epsilon,p},\nu_{\hg}) &=& \text{const.}-\langle a_{\epsilon,p},\cdot\rangle& \text{on } \partial B_1.
\end{array}
\right.
\end{equation}

For each $\epsilon>0$ we define the homothetic metric
$g_\epsilon=\epsilon^{-2}g$ on $M$. Then, for each $p\in M$ and
$\epsilon\in(0,\epsilon_0)$, we consider the $\C^{2,\alpha}$-domain $$
\Omega_{\epsilon,p}:=B_{\epsilon(1+v_{\epsilon,p})}^g(p)=B_{1+v_{\epsilon,p}}^{g_\epsilon}(p)
$$
of $M$, which is centered at $p$. Using the parametrization $Y\colon B_1\to \Omega_{\epsilon,p}$ introduced in~\eqref{eq:Y}, we can define the function $u_{\epsilon,p}\in\C^{2,\alpha}\Dir(\Omega_{\epsilon,p})$ by means of the relation $\hu_{\epsilon,p}=:u_{\epsilon,p}\circ Y$. Recalling that $\hg=\epsilon^{-2}Y^*g=Y^*g_\epsilon$ and $\hf(y,z)=f(Y(y),z)$, we get that 
\[
\Delta_{\hg} \hu_{\epsilon,p}+\hf(\cdot,\hu_{\epsilon,p})=\Delta_{Y^*g_\epsilon}(u_{\epsilon,p}\circ Y)+f(Y(\cdot),u_{\epsilon,p}\circ Y)=\Delta_{g_\epsilon} u_{\epsilon,p}+f(\cdot,u_{\epsilon,p}),
\]
and
\[
\vol_g(B^g_{\epsilon(1+v_{\epsilon,p})}(p))=\epsilon^n\vol_{\bar{g}}(B_{1+v_{\epsilon,p}})=\epsilon^n\vol_{\hg}(B_1)=\epsilon^n|B_1|=|B_\epsilon|.
\]
Hence, we can reformulate what we have obtained as follows. For each $\epsilon\in(0,\epsilon_0)$ and $p\in M$, we have a collection of $\C^{2,\alpha}$-domains $\{\Omega_{\epsilon,p}\}_{p\in M}$ of $M$, where $\Omega_{\epsilon,p}=B^{g_\epsilon}_{1+v_{\epsilon,p}}(p)$ is centered at $p$ and $\vol_{g}(\Omega_{\epsilon,p})=|B_\epsilon|$, as well as a collection of solutions $\{u_{\epsilon,p}\}_{p\in M}$ to the problem
\begin{equation}\label{eq:onp_affine}
\left\{
\begin{array}{rcll}
\Delta_{g_\epsilon} u_{\epsilon,p}+f(\cdot,u_{\epsilon,p})&=&0 & \text{in } \Omega_{\epsilon,p}
\\
u_{\epsilon,p}&>&0 & \text{in } \Omega_{\epsilon,p}
\\
u_{\epsilon,p}&=&0 & \text{on } \partial\Omega_{\epsilon,p}
\\
g_{\epsilon}(\nabla_{g_\epsilon} u_{\epsilon,p},\nu_{g_\epsilon}) &=& \text{const.}-\langle a_{\epsilon,p},Y^{-1}(\cdot)\rangle& \text{on } \partial\Omega_{\epsilon,p},
\end{array}
\right.
\end{equation}
where $\nu_{g_{\epsilon}}=Y_*\nu_{\hg}$ is the outward $g_\epsilon$-unit normal vector field to $\partial\Omega_{\epsilon,p}$. 

Now we are interested in finding some point $p\in M$ such that the
vector $a_{\epsilon,p}\in T_pM$, provided by
Proposition~\ref{prop:affine} and appearing in the Neumann condition
in~\eqref{eq:onp_affine}, is zero. Note that, as $\Delta_{g_{\epsilon}}=\epsilon^2\Delta_g$, this would be equivalent
to saying that $u_{\ep,p}$ is a solution to the overdetermined
problem~\eqref{eq:onp} in the domain $\Omega_{\epsilon,p}$ with $\la:=\epsilon^{-2}$. 

To show the existence of zeros of the vector field $a_\ep$ we will use a
variational technique. Given a function $u\in H^1_0(\Om)$, we define its energy as
\[
J(u):=\int_{\Omega}\left(\frac{1}{2}\left| \nabla_g u\right|^2_g-F(\cdot,u)\right)dV_g,
\]
where 
$$
F(p,z):=\int_0^z f(p,\zeta)\, d\zeta.
$$

Let $\Omega_0$ be a smooth bounded domain of $M$. We call $\{\Omega_t\}_{t\in(0,t_0)}$ a deformation of $\Omega_0$ if there exists a smooth vector field $\Xi$ such that $\Omega_t=\xi(t,\Omega_0)$, where $\xi(t,\cdot)$ is the flow of $\Xi$.

\begin{proposition}\label{prop:variational}
	Assume that $\{\Omega_t:=\xi(t,\Om_0)\}_{t\in(-t_0,t_0)}$ is a
        deformation of a smooth bounded domain $\Omega_0$, and $u_t\in
        C^2(\Omega_t)$ is a solution to the Dirichlet problem
        $\Delta_g  u_t+f(\cdot, u_t)=0$ in $\Omega_t$, $u_t=0$ on
        $\partial \Omega_t$, depending smoothly on $t$. Then the
        derivative of the energy functional $J(t):=J(u_t)$ at $t=0$
        reads as
	\[
	J'(0)=-\frac{1}{2}\int_{\partial \Omega_0} g(\nabla_g u_0,\nu_0)^2 g(\Xi,\nu_0)dV_g,
	\]
	where $\nu_0$ is the outward unit normal vector field on $\partial \Omega_0$.
\end{proposition}
\begin{proof}
	Using the differentiation formula for moving regions (see \cite[Appendix~C.4]{Evans}), we get
	\begin{align*}
		J'(0)={}&\int_{\Omega_0}g(\nabla_g \partial_t u_0,\nabla_g u_0)dV_g-\int_{\Omega_0}f(\cdot,u_0)\partial_t u_0dV_g
		\\
		&+\int_{\partial \Omega_0}\left(\frac{1}{2}g(\nabla_g u_0,\nu_0)^2-F(\cdot, u_0)\right)g(\Xi,\nu_0)d\sigma_g,
	\end{align*}
	where $\partial_t u_0$ stands for the derivative of $u_t$ with
        respect to $t$ at $t=0$, $\Xi$ is the vector field defining
        the deformation, and $d\sigma_g$ is the induced Riemannian
        area measure on $\partial \Omega_0$. Integrating by parts in the first integral and using the relation $\Delta_g u_0+f(\cdot, u_0)=0$, we have
	\begin{align*}
		\int_{\Omega_0}g(\nabla_g \partial_t u_0,\nabla_g u_0)dV_g={}&-\int_{\Omega_0}\partial_t u_0 \, \Delta_g u_0dV_g +\int_{\partial\Omega_0}\partial_t u_0 g(\nabla_g u_0,\nu_0)d\sigma_g
		\\
		={}&\int_{\Omega_0}f(\cdot, u_0)\partial_t u_0dV_g +\int_{\partial\Omega_0}\partial_t u_0 g(\nabla_g u_0,\nu_0)d\sigma_g.
	\end{align*}
	Denoting by $\xi$ the flow generated by $\Xi$, by definition we have that $u_t(\xi(t,p))=0$ for all $p\in\partial \Omega_t$. Differentiating this identity with respect to $t$ at $t=0$ we obtain $\partial_t u_0=-g(\nabla_g u_0,\Xi)$ on $\partial \Omega_0$. But, since $u_0$ is constant on $\partial \Omega_0$, we can write $\partial_t u_0=-g(\nabla_g u_0,\nu_0)g(\Xi,\nu_0)$ on $\partial \Omega_0$. Moreover, $F(\cdot, u_0)=F(\cdot,0)=0$ on $\partial \Omega_0$. Altogether, we obtain the formula in the statement.
\end{proof}

Now, using the solutions $u_{\ep,p}$ to the Dirichlet problem on the domain $\Omega_{\epsilon,p}$ to~\eqref{eq:onp_affine}, for each $\epsilon>0$ small enough we can define a smooth function $\cal{J}_\epsilon\colon M\to\R$ as
\[
\cal{J}_\epsilon(p):=J(u_{\epsilon,p})=\int_{\Omega_{\epsilon,p}}\left(\frac{1}{2}\left| \nabla_{g_\epsilon} u_{\epsilon,p}\right|^2_{g_\epsilon}-F(\cdot,u_{\epsilon,p})\right)dV_{g_\epsilon}, \quad p\in M.
\]
The following proposition provides a key result to derive the main theorems of this paper:

\begin{proposition}\label{prop:critical_point}
For each small enough $\epsilon>0$, $u_{\ep,p}$ is a solution to the
overdetermined problem~\eqref{eq:onp} on the domain $\Omega_{\epsilon,p}$ with $\la:=\ep^{-2}$ if and only if $p$ is a critical point of $\cal{J}_\epsilon$.
\end{proposition}
\begin{proof}
    In view of~\eqref{eq:onp_affine}, $u_{\ep,p}$ is a solution to the
overdetermined problem~\eqref{eq:onp} on the domain $\Omega_{\epsilon,p}$ with $\la:=\ep^{-2}$ if and only if $a_{\epsilon,p}=0$.
	We start by calculating the differential of
        $\cal{J}_\epsilon$. Let $w \in T_p M$ and define
        $q_t=\exp_p^{g_\epsilon}(tw)$. If $t$ is small enough,
        $\partial \Omega_{\epsilon,q_t}$ can be expressed as a graph
        of a function $h_t$ over $\partial \Omega_{\epsilon,p}$. As
        the domains have fixed volume, this yields a volume-preserving vector field $\Xi$ on $\partial \Omega_{\epsilon, p}$ given by 
	\[
	\Xi=\left.\frac{\partial h_t}{\partial t}\right|_{t=0}\nu_{g_\epsilon}.
	\]
	Then, by Proposition~\ref{prop:variational} and \eqref{eq:onp_affineB1} we have
	\begin{align}\nonumber
	D_p\cal{J}_\epsilon(w)&=\left.\frac{d}{dt}\right|_{t=0}\cal{J}_\epsilon(q_t)=-\frac{1}{2}\int_{\partial \Omega_{\epsilon,p}}g_\epsilon(\nabla_{g_\epsilon} u_{\epsilon,p},\nu_{g_\epsilon})^2 g_\epsilon(\Xi,\nu_{g_\epsilon}) d\si_{g_\epsilon}
	\\ 	\label{eq:diff_J}
	&= -\frac{1}{2}\int_{\partial B_1}\left(b_{\epsilon,p}-\langle a_{\epsilon,p},\cdot\rangle\right)^2\hg(\widehat{\Xi},\nu_{\hg})  d\si_{\hg}
	\end{align}
	where $b_{\epsilon,p}$ is a constant, $\widehat{\Xi}= Y^*\Xi$ and $\nu_{\hg}= Y^*\nu_{g_\epsilon}$. 
	Observe that if $a_{\epsilon,p}=0$, then $p$ is a critical point of~$\cal{J}_\epsilon$, since $\widehat{\Xi}$ is volume-preserving.
	
	In order to prove the converse, assume that $D_p\cal{J}_\epsilon=0$. Then \eqref{eq:diff_J} implies
	\begin{equation}\label{eq:diff_J_0}
	2b_{\epsilon,p} \int_{\partial B_1}\langle a_{\epsilon,p},\cdot\rangle\hg(\widehat{\Xi},\nu_{\hg})  d\si_{\hg}=\int_{\partial B_1}\langle a_{\epsilon,p},\cdot\rangle^2\hg(\widehat{\Xi},\nu_{\hg})  d\si_{\hg},
	\end{equation}
	for all $w\in T_p M$. By the Taylor expansion of the metric,
        for all $\epsilon$ small enough there exists a constant $c$,
        independent of~$\ep$, such that
	\[
	\left|\hg(\widehat{\Xi},\nu_{\hg})-\langle w,\cdot\rangle\right|\leq c\epsilon |w|.
	\]
	If we choose $w:=b_{\epsilon,p} a_{\epsilon,p}$ we have that 
$$
\hg(\widehat{\Xi},\nu_{\hg})=b_{\epsilon,p}\langle
a_{\epsilon,p},\cdot\rangle +\epsilon k\,,
$$
where 
$$
|k|\leq c|b_{\epsilon,p}||a_{\epsilon,p}|\,.
$$
Using this in \eqref{eq:diff_J_0} we deduce that there exists a constant $C>0$ (independent of~$\epsilon$) such that for all $\epsilon$ small enough the following inequality holds:
	\[
	2b_{\epsilon,p}^2 \int_{\partial B_1}\langle a_{\epsilon,p},\cdot\rangle^2 dV_{\hg} \leq C |b_{\epsilon,p}|\left(\epsilon|a_{\epsilon,p}|^3+|a_{\epsilon,p}|^3+\epsilon|a_{\epsilon,p}|^2\right).
	\]
	Since for $\epsilon$ small enough the left hand side is bounded from below by a positive constant times $b_{\epsilon,p}^2|a_{\epsilon,p}|^2$, we obtain
	\begin{equation}\label{eq:inequal_a}
	b^2_{\epsilon,p}|a_{\epsilon,p}|^2\leq C|b_{\epsilon,p}|\left(\epsilon|a_{\epsilon,p}|+|a_{\epsilon,p}|+\epsilon\right)|a_{\epsilon,p}|^2,
	\end{equation}
	for some other constant $C>0$. Now, as $b_{\epsilon,p}\neq 0$ when
        $\epsilon=0$ (because, by Assumption~A(i), $\partial_\nu\phi_p\rvert_{\partial B_1}<0$), we have that $|b_{\epsilon,p}|>C$ for some uniform constant $C>0$. Since $\lim_{\epsilon\to 0}a_{\epsilon,p}=0$, the inequality \eqref{eq:inequal_a} implies that, for $\epsilon$ small enough, $a_{\epsilon,p}=0$. Finally observe that the existence of an $\epsilon_0>0$ independent of $p\in M$ is guaranteed by the compactness of $M$ and the fact that the constants in the arguments above are bounds on quantities that only depend on the norms of the geometric objects involved.	
\end{proof}

This completes the proof of Theorem~\ref{th:main2}.


\section{Variations and applications of the proof of Theorem~\ref{th:main}}
\label{S.variations}

In this section we derive some corollaries of the proof of Theorem~\ref{th:main2}. In particular, we prove Theorem~\ref{th:Eulerchar} and a uniqueness result that will be important in the following sections.

\begin{proof}[Proof of Theorem~\ref{th:Eulerchar}]
Let us now replace the equation $\Delta_g u+\lambda f(\cdot, u)=0$ by
the more general equation
\begin{equation}\label{eqX}
\Delta_g u+\langle \nabla_g u, X\rangle+\lambda f(\cdot, u)=0\,,
\end{equation}
where $X$ is a $C^{\al}$-vector field on the manifold. Our goal now is
to show that if Assumption~A holds (notice that this assumption
does not involve the vector field~$X$) and the Euler characteristic
of~$M$ is nonzero, then the overdetermined boundary problem for
Equation~\eqref{eqX} admits a nontrivial solution on domains that are small perturbations of small
geodesic balls.

Going through the proof of Theorem~\ref{th:main2} one easily checks that
the analysis of the rescaled equation developed in Step~1 carries over to the more general
equation~\eqref{eqX}, as does the discussion of the behavior of the
normal derivative in Step~2. However, the proof of the fact that the vector
field~$a_\ep$ has zeros, presented in Step~3, does not apply to this
more general equation, as the demonstration relies on the variational
formulation of the equation $\Delta_g u+\lambda f(\cdot,
u)=0$. However, since we have already shown that $a_\ep$ is a
$C^{1,\al}$-vector field on~$M$ (Proposition~\ref{prop:affine}),
$a_\ep$ is guaranteed to have zeros if $M$ is a compact manifold of
nonzero Euler characteristic. Theorem~\ref{th:Eulerchar} then follows.
\end{proof}

The construction method we have developed implies the following
uniqueness result:

\begin{proposition}\label{cor:unique}
	Suppose that the nonlinearity $f$ satisfies
        Assumption~A. Suppose that we have two solutions to the
        overdetermined problem~\eqref{eq:onp}, given by $(u,
        \Om:=B^g_{\epsilon(1+v)}(p), \la:=\ep^{-2})$ and $(u',
        \Om':=B^g_{\epsilon(1+v')}(p), \la:=\ep^{-2})$,
        respectively. Then, if
	\begin{enumerate}
		\item  $\vol_g \bigl(\Omega\bigr)=\vol_g \bigl(\Omega'\bigr)$,
		\item $\Omega$ and $\Omega'$ are centered at $p$,
		\item $\epsilon$ and the $\C^{2,\alpha}$-norms of $v$,
                  $v'$, $\hu-\phi_p$ and $\hu'-\phi_p$ are small
                  enough (i.e., smaller than some quantity~$\de$
                  depending on~$p$),
	\end{enumerate}
	then $u=u'$ and $\Om=\Om'$. Moreover, if $M$ is compact, then
        $\delta>0$ can be taken to be independent of the point $p\in M$.
\end{proposition}
\begin{proof}
	In the proof of Proposition~\ref{prop:implicit} we could have introduced a new variable $\eta\in(-1,1)$ in the definition of the map $\cal{N}$, and substituted the second component of $\cal{N}$ by $\vol_{\hg}(B_1)-|B_{1+\eta}|$. Then, the implicit function theorem would apply in the same way at the point $(p,0,0,0,0,\phi_p)$ to derive the existence of unique $v_0=v_0(p,\epsilon,\eta,\bar{v})\in \R$ and $\hu=\hu(p,\epsilon,\eta,\bar{v})\in\C^{2,\alpha}\Dir(B_1)$ in small enough neighborhoods of $0\in\R$ and of $\phi_p\in\C^{2,\alpha}\Dir(B_1)$, respectively, satisfying $\cal{N}(p,\epsilon,\eta,\bar{v},v_0,\hu)=(0,0)$, whenever $\epsilon$, $\eta$ and $\bar{v}$ have small enough norm. Then, the remaining arguments in Section~\ref{sec:setup} are valid only with some formal modifications (e.g.\ the operator $\cal{F}$ should now be defined as a function of $(p,\epsilon,\eta,\bar{v})$, and Proposition~\ref{prop:affine} would yield $a_{\epsilon,\eta,p}$ and $\bar{v}_{\epsilon,\eta,p}$). In particular, the arguments in Proposition~\ref{prop:affine} imply the uniqueness of $a_{\epsilon,\eta,p}$ and $\bar{v}_{\epsilon,\eta,p}$ in small enough neighborhoods of $0$ in $T_pM$ and $\C^{2,\alpha}\avg(S^{n-1})$, respectively, whenever $\epsilon$ and $\eta$ are small enough, say less than~$\delta$. Hence, the statement of the corollary holds if, in addition to the hypotheses, we have $\vol_g \bigl(B^g_{\epsilon(1+v)}(p)\bigr)=\vol_g \bigl(B^g_{\epsilon(1+v')}(p)\bigr)=|B_{\epsilon(1+\eta)}|$, for some $\eta\in(-\delta,\delta)$. But if $\epsilon$, $\|\bar{v}\|_{\C^{2,\alpha}(S^{n-1})}$ and $\|\bar{v}'\|_{\C^{2,\alpha}(S^{n-1})}$ are small enough, there exists $\eta\in(-\delta,\delta)$ satisfying the additional condition on the volumes.
\end{proof}

\section{Overdetermined problems on noncompact manifolds}\label{sec:homogeneity}

The purpose of this section is to discuss the possibility of relaxing
the compactness assumption on $M$ in the results above. Thus, in the
rest of this section, $M$ is not assumed to be compact.

A first result is that Theorem~\ref{th:main} holds also if the compactness assumption on $M$ is
substituted by a suitable homogeneity assumption. To present a
rigorous formulation of this fact, consider a group $\G$ of
isometries of $(M,g)$ that leave $f$ invariant, namely
\[
\G\subseteq\{\varphi\in\mathrm{Isom}(M): f(\varphi(p),z)=f(p,z),\text{ for all } p\in M, z\in\R\}.
\]
Consider the orbit space $M/\G$, that is, the set of all the
orbits
\[
\G\cdot p=\{\varphi(p):\varphi\in \G\}\,,
\]
with $p\in M$, of the
isometric $\G$-action on $M$, endowed with the quotient topology
induced by the canonical projection map $M\to M/\G$, $p\mapsto \G\cdot
p$ (see~\cite[\S2.1]{BCO2} for more information on isometric
actions). Notice that $M/\G$ does not have to be a differentiable manifold: it can be a
manifold with corners, or even an orbifold. 

The fact that Theorem~\ref{th:main} remains valid if the orbit space $M/\G$ is compact
(but not necessarily $M$) follows from the following proposition. Of
course, particularly simple yet important
examples of problems to which this result applies are noncompact
homogeneous spaces (or, more generally, spaces with a co-compact
isometry group) with a position-independent nonlinearity $G(z)$. Recall that a Riemannian manifold $M$ is called homogeneous if its isometry group acts transitively on $M$ or, equivalently, if given any two points $p$, $q\in M$ there exists an isometry of $M$ mapping $p$ to $q$.

\begin{proposition}\label{th:quotient}
		 Theorem~\ref{th:main2} and Proposition~\ref{cor:unique} remain true if the compactness
		 assumption on $M$ is substituted by the following two hypotheses:
		 \begin{enumerate}[{\rm (a)}]
		 	\item the orbit space $M/\G$ is compact,
		 	\item the map $\phi_p$ in	Assumption~A is $\G$-invariant (i.e., $\phi_p=\phi_{\vp(p)}$ for all $\vp\in \G$). 
		 \end{enumerate}
		 Moreover, Theorems~\ref{T.ex1} and~\ref{T.ex2}, and hence Theorem~\ref{th:main}, remain true if the compactness assumption on $M$ is substituted by condition \rm{(a)}.
\end{proposition}
\begin{proof}
	By definition, the maps $\cal{N}$ and $\cal{F}$ considered in
        Section~\ref{sec:setup} are invariant under the isometries in
        $\G$. Thus, the compactness arguments used in
        Section~\ref{sec:setup} (only in the proofs of
        Propositions~\ref{prop:implicit} and~\ref{prop:affine}) can be
        carried out under the assumption that $M/\G$ is compact and $\phi_p$ is $\G$-invariant. Also,
        by construction, the objects involved in the proof of
        Proposition~\ref{prop:critical_point} (namely
        $\cal{J}_\epsilon$, $a_{\epsilon,p}$, $b_{\epsilon_p}$) are
        $\G$-invariant, and hence the proof carries over to the case
        that $M/\G$ is compact. Similarly, the proofs of
        Theorems~\ref{T.ex1} and~\ref{T.ex2} remain valid under the hypothesis (a) since $\cal{G}$ is $\G$-invariant, giving rise to a $\G$-invariant function $\phi_p$ satisfying Assumption~A.
	
	Moreover, compactness was needed to ensure the existence of critical points of the functions $\cal{J}_\epsilon$. But, since $\cal{J}_\epsilon$ is invariant under $\G$, it induces a continuous map $\widetilde{\cal{J}}_\epsilon\colon M/\G\to\R$. Any preimage under the projection map $M\to M/\G$ of a maximum point for $\widetilde{\cal{J}}_\epsilon$ is also a maximum point for $\cal{J}_\epsilon$, and hence, a critical point for $\cal{J}_\epsilon$. This shows that Theorem~\ref{th:main2} holds under the assumptions (a) and (b). Combining this with the generalization of Theorems~\ref{T.ex1} and~\ref{T.ex2}, we deduce that Theorem~\ref{th:main} holds under condition (a).  
\end{proof}

\begin{remark}
	With the same arguments as above one can show that Theorems~\ref{th:main}, \ref{T.ex1}, \ref{T.ex2} and~\ref{th:main2}, and Proposition~\ref{cor:unique}, hold, not only for homogeneous spaces, but also for locally uniformly homogeneous spaces, that is, for those Riemannian manifolds $(M,g)$ such that there exist $R>0$ satisfying that, for any $p$, $q\in M$, there exists an isometry $\varphi\colon B^g_R(p)\to B^g_R(q)$. In this case, one has to restrict to position-independent nonlinearities $f(p,z)=G(z)$.
\end{remark}

Next we shall show how to extend the previous results to a broader
family of noncompact manifolds that do not need to have any global
isometries and whose nonlinearities can depend on the point in
the manifold, but for which, roughly speaking, there is a kind of
homogeneous behavior at infinity. The geometry at infinity of the manifold $M$ will be modeled by some homogeneous Riemannian manifold $M_\infty$ with metric $g_\infty$. In order to formalize this idea,
let us introduce the appropriate terminology.

\begin{definition}\label{def:asymp}
	We will say that the Riemannian $n$-manifold $(M,g)$ is
        \emph{asymptotically homogeneous} if there exists an
        $n$-dimensional homogeneous Riemannian manifold~$M_\infty$ with metric $g_\infty$, a point $p_\infty\in
        M_\infty$, $R>0$, and a
        function $\vp:M\times B_R^{g_\infty}(p_\infty)\to M$ with the
        property that for each $\delta>0$ there
        exists a relatively compact domain $M_\delta$ of~$M$
        such that for all $p\in M\setminus \overline{M}_\delta$, $\vp_p:=\vp(p,\cdot)$ defines a $\C^{2,\alpha}$-diffeomorphism
$\varphi_p\colon B^{g_\infty}_R(p_\infty)\to B^{g}_R(p)$
and
$\|\varphi_p^*g-g_\infty\|_{\C^{2,\alpha}(B^{g_\infty}_R(p_\infty))}<\delta$.
\end{definition}

Observe that, with this definition, the injectivity radius of an asymptotically homogeneous
manifold is always larger than some positive constant $R$.

\begin{remark}
	The notion we have just introduced can be regarded as a broad
        generalization of other well-known concepts. For example, if
        the model manifold $(M_\infty,g_\infty)$ in
        Definition~\ref{def:asymp} is a Euclidean space, then the
        above definition means that $(M,g)$ is asymptotically
        flat. Notice that, although there are several definitions of
        asymptotically flat manifolds in the literature, they usually
        imply the notion of asymptotically homogeneous manifold, as is
        the case, for instance, of Schoen and Yau's
        definition~\cite{SY:cmp}. Similarly, the notion of
        asymptotically hyperbolic manifold (see
        e.g.~\cite{JS:acta}) corresponds to taking as model manifold
        $(M_\infty,g_\infty)$ the real hyperbolic space. In any case, the choice of the point $p_\infty$ is
        inessential due to the isometries of~$(M_\infty,g_\infty)$.
\end{remark}

Now we need to define an analogous notion of ``asymptotic
homogeneity'' for the nonlinearity. To this end we will make the
following assumption:

\begin{assumptionB}	
$M$ is an asymptotically homogeneous manifold as above, the nonlinearity $f\in C^{1,\al}_{\mathrm{loc}}(M\times\R)$ satisfies Assumption~A with radial functions $\phi_p\in\C^{2,\alpha}\Dir(B_1)$ and constant $\bar{\lambda}$, 
and additionally there exist functions $G_\infty\in\C^{1,\alpha}_{\mathrm{loc}}(\R)$ and $\phi_\infty\in\C^{2,\alpha}\Dir(B_1)$ such that:
\begin{enumerate}
\item For each $\delta>0$ there exists a relatively compact 
  domain $M_\delta$ of $M$ such that 
\[
\|f(p,\cdot)-G_\infty\|_{\C^{1,\alpha}((-N,N))}+\|\phi_p-\phi_\infty\|_{\C^{2,\alpha}(B_1)}<\delta
\]
for all $p\in M\setminus\overline{M}_\delta$, where $N>0$ is such that, for all $p\in M$, 
\[
\|\phi_p\|_{L^\infty}+\|\phi_\infty\|_{L^\infty}<\frac{N}{2}.
\]
		\item The function $G_\infty$ also satisfies Assumption~A with the same constant $\bar{\lambda}$, the corresponding radial function on the unit ball being precisely $\phi_\infty$.
                \end{enumerate}
              \end{assumptionB}

              The main result in this section  is that
              Theorem~\ref{th:main2} remains valid for asymptotically
              homogeneous manifolds with asymptotically homogeneous
              nonlinearities. It is clear that, just as before, the
              key is to prove the result under Assumption~B. 
This is what we will do in the proof of the following theorem. 

\begin{theorem}\label{th:main_asymp}
Suppose that Assumption~B holds. Then for every small enough positive $\ep$ there exists a domain $\Om\subset M$ with volume equal to $|B_\ep|$ and a positive constant $\la$, which is of order $\ep^{-2}$, such that the overdetermined problem~\eqref{eq:onp}
admits a nontrivial nonnegative solution. The domain $\Omega$ is a
$\C^{2,\alpha}$-small perturbation of a geodesic ball of
radius~$\ep$.
\end{theorem}

Before proving this result, we state and prove the following consequence of Theorem~\ref{th:main_asymp}. Roughly speaking, it guarantees that Theorem~\ref{th:main_asymp} can be applied to the appropriate modifications of the two types of nonlinearities studied in Section~\ref{S.ass}.

\begin{theorem}\label{th:asympt_types}
	Let $(M,g)$ be an asymptotically homogeneous manifold. Let $f$ be a function in $C^{1,\al}\loc(M\times\RR)$ such
	that,  as $p\to \infty$,
	$f(p,z)$ tends to a certain function $G_\infty(z)$ locally in the
	$C^{1,\al}$-norm in the sense that for any $\delta>0$ and any compact interval
	$I\subset\RR$ there exists a relatively compact domain $M_{I,\delta}$ of $M$ such that
	\[
	\|f(p,\cdot)-G_\infty\|_{C^{1,\al}(I)}<\delta,
	\]
	for all $p\in M\setminus \overline{M}_{I,\delta}$.
	Suppose that one of the following assumptions on $f$ holds
	\begin{enumerate}[{\rm (i)}]
		\item $\displaystyle\inf_{p\in M} f(p,0)>0$, or
		\item $f\in C^{2}\loc(M\times\R)$ converges to $G_\infty\in C^2\loc(\R)$ locally in the $C^2$-norm, $f(p,0)=0$,  $f_z(p,0)=c>0$ and $\displaystyle\inf_{p\in M}|f_{zz}(p,\cdot)|> 0$ for all~$p\in M$, where $c$ is a constant independent of $p$.
	\end{enumerate}
	
	Then for every small enough~$\ep$
	there exists  a domain $\Om\subset M$ with volume equal to $|B_\ep|$ and a positive constant $\la$, which is of order $\ep^{-2}$, such that the overdetermined problem~\eqref{eq:onp}
	admits a nontrivial nonnegative solution. The domain $\Omega$
        is a
	$\C^{2,\alpha}$-small perturbation of a geodesic ball of
	radius~$\ep$.
\end{theorem}

\begin{proof}
	We just have to prove that the nonlinearity $f$ satisfies Assumption~B for each one of the cases (i) and (ii) in the statement.
	
Since $f$ tends to $G_\infty$ locally in the $C^{1,\al}$-norm (or in the $C^2$-norm for case (ii)), it follows that $G_\infty$ satisfies one of the two conditions analyzed in Section~\ref{S.ass}. Thus, Theorems~\ref{T.ex1} and~\ref{T.ex2} applied to the nonlinearity $G_\infty$ (which is independent of $p\in M$) guarantee the existence of $\phi_{\infty,\bar{\lambda}}\in C^{2,\al}\Dir(B_1)$ satisfying Assumption~A, with nonlinear eigenvalue $\bar{\lambda}>0$ in an open interval with endpoint $0$ for case (i) or $\lambda_{0}=\lambda_1(B_1)/c$ for case (ii).

In the proofs of Theorems~\ref{T.ex1} and~\ref{T.ex2}, the only point where we needed the compactness assumption was at the very end of the proofs, in order to guarantee that the smallness conditions imposed on $\bar{\lambda}$ (for case (i)) or $|\bar{\lambda}-\lambda_0|$ (for case (ii)) were uniform in $p\in M$. Now $M$ is no longer compact, but the function $f$ tends to a function $G_\infty$ that satisfies Assumption~A. Thus, using the continuous dependence in the implicit function theorem and the Crandall-Rabinowitz theorem used in Theorems~\ref{T.ex1} and~\ref{T.ex2}, respectively, both results also hold in the current setting. That is, the nonlinearity $f$ satisfies Assumption~A for each $\bar{\lambda}>0$ in an open interval (with endpoint $0$ for case (i) or $\lambda_{0}$ for case (ii)) and $\phi_{p,\bar{\lambda}}\in C^{2,\al}\Dir(B_1)$ smoothly depending on $p\in M$. Moreover, by construction, the functions $\phi_{\infty,\bar{\lambda}}$ and $\phi_{p,\bar{\lambda}}$ have small $C^{2,\al}$-norm. Hence, we can take $N>0$ such that $\|\phi_{p,\bar{\lambda}}\|_{L^\infty}+\|\phi_{\infty,\bar{\lambda}}\|_{L^\infty}<N/2$, for all $p\in M$ and all $\bar{\lambda}$ in some open interval.

 The fact that
$\|f(p,\cdot)-G_\infty\|_{\C^{1,\alpha}((-N,N))}<\de$ for all $p$ outside a
compact subset $\overline{M}_\delta$ of~$M$ follows directly from the hypothesis of locally uniform $C^{1,\al}$-convergence
of $f(p,\cdot)$ to $G_\infty$. In view of this convergence, the fact that $\|\phi_{p,\bar{\lambda}}-\phi_{\infty,\bar{\lambda}}\|_{\C^{2,\alpha}(B_1)}$ 
is also small follows again from the continuous dependence in the arguments used in Theorems~\ref{T.ex1} and~\ref{T.ex2}, respectively. This shows that $f$ satisfies Assumption~B for each $\bar{\lambda}$ in some open interval.
\end{proof}

Notice that Theorem~\ref{T.asympt} in the Introduction is the particular case of the previous Theorem~\ref{th:asympt_types} for $(M,g)=(\R^n,g_0)$, where $g_0$ is the Euclidean metric. Clearly, the Euclidean space is a homogeneous Riemannian manifold so, in particular, it is asymptotically homogeneous. Indeed, in this case, in the definition of asymptotic homogeneity one can take $(M_\infty,g_\infty)$ to be the Euclidean space~$(\RR^n,g_0)$ and $B^{g_\infty}_R(p_\infty)$ to be the unit ball
centered at the origin. The function~$\vp:\RR^n\times B\to\RR^n$ is
\[
\vp(p,z):=p+z\,,
\]
so $\varphi_p=\varphi(p,\cdot)$ is an isometry onto its image, for each $p\in\R^n$. Similarly, Theorem~\ref{th:asympt_types} applies to all noncompact homogeneous Riemannian manifolds, including, as a very specific but important case, the hyperbolic spaces. Other important examples of manifolds to which Theorem~\ref{th:asympt_types} applies are the symmetric spaces of noncompact type, as well as all known examples of noncompact harmonic spaces (see Section~\ref{sec:symmetry} for more information on these spaces).

\medskip
Now we proceed with the proof of Theorem~\ref{th:main_asymp}, where we include a couple of propositions.

\begin{proof}[Proof of Theorem~\ref{th:main_asymp}]
We will follow the ideas of the proof of Theorem~\ref{th:main2}. We start by adapting the proof of Proposition~\ref{prop:implicit} to the new setting. Notice that the difficulty is to show that it is possible to take a bound $\delta$ in Proposition~\ref{prop:implicit} that does not go to zero as $p\in M$ goes to infinity.

\begin{proposition}\label{prop:implicit_asymp}
	The claims in Proposition~\ref{prop:implicit} hold also in the
        case that the Riemannian manifold $M$ is asymptotically homogeneous.
\end{proposition}
\begin{proof}
Let $\cal{M}$ denote the convex open cone of Riemannian metrics inside
the Banach space of all $\C^{2,\alpha}$-symmetric tensors on the geodesic ball $B:=B_R^{g_\infty}(p_\infty)$ given in Definition~\ref{def:asymp}. We will employ the usual notations; for instance, given $h\in \cal{M}$ we define $\bar{h}=\epsilon^{-2}(\exp^h_{p_\infty}\circ\, T_\epsilon)^*h$, we consider the parametrization $Y\colon B_1\subset T_{p_\infty} B\to B_{\epsilon(1+v)}^{h}(p_\infty)$ as in~\eqref{eq:Y}, the metric $\widehat{h}=\beta^*\bar{h}$ on $B_1$, and the function $\widehat{s}(y,z)=s(Y(y),z)$, $y\in B_1$, $z\in\R$. 
Consider the smooth map
\[
\begin{array}{r@{\hspace{-0.05ex}}c@{\hspace{-0.01ex}}cl}
	\cal{N}\colon& \cal{M}'\times\C^{1,\alpha}(B\times(-N,N))\times\cal{U}&\to&\C^{0,\alpha}(B_1)\times\R
	\\
	&(h,s,\epsilon,\bar{v},v_0,\psi) & \mapsto & \bigl(\Delta_{\widehat{h}}\psi+\widehat{s}(\,\cdot\,,\psi),\;\vol_{\widehat{h}}(B_1)-|B_1|\bigr),
\end{array}
\]
where $\cal{M}'$ is an open neighborhood of $g_\infty$ in $\cal{M}$ and $\cal{U}$ is an open neighborhood of $(0,0,0,0)$ in
$[0,\infty)\times\C^{2,\alpha}\avg(S^{n-1})\times\R\times C\Dir^{2,\alpha}(B_1)$ such that
$\epsilon(1+\bar{v}+v_0)<R$  and $\|\psi\|_{L^\infty}<N$, for all $h\in \cal{M}'$ and $(\epsilon,\bar{v},v_0,\psi)\in\cal{U}$; recall that $N>0$ was introduced in Assumption~B.

 Just as in the
proof of Proposition~\ref{prop:implicit}, we can apply the implicit
function theorem to $\cal{N}$ at the point
$(g_\infty,G_\infty,0,0,0,\phi_\infty)$, where $g_\infty$ is the
metric of $M_\infty$, $G_\infty$ and $\phi_\infty$ are as in
Assumption~B, and $G_\infty$ is to be understood as a function defined
on $B\times\R$, restricted to $B\times (-N,N)$, but which is independent of the point of~$B$. Therefore, for each $(h,s,\epsilon,\bar{v})$ in a
neighborhood~$\cal{V}$ of $(g_\infty,G_\infty,0,0)$ in $\cal{M}\times
\C^{1,\alpha}(B\times(-N,N))\times[0,\infty)\times
\C^{2,\alpha}\avg(S^{n-1})$, there exists a unique $(v_0,\hu)$,
smoothly depending on $(h,s,\epsilon,\bar{v})$, in a neighborhood of
$(0,\phi_\infty)$ in $\R\times\C\Dir^{2,\alpha}(B_1)$ such that
$\cal{N}(h,s,\epsilon,\bar{v},v_0,\hu)=0$. We can assume that for some
$\delta>0$ the neighborhood $\cal{V}$ is given by those
$(h,s,\epsilon,\bar{v})$ satisfying 
\[
\max\{\|h-g_\infty\|_{\C^{2,\alpha}(B)}, \;
\|s-G_\infty\|_{\C^{1,\alpha}(B\times(-N,N))}, \; \epsilon,\; 
\|\bar{v}\|_{\C^{2,\alpha}(S^{n-1})}\}<\delta\,.
\]

For this $\delta$, let $M_\delta$ be a relatively compact domain in
$M$ that satisfies simultaneously the condition stated in
Definition~\ref{def:asymp} and item~(i) in Assumption~B, and let
$M'_\delta\supset M_\delta$ be another relatively compact domain such
that if $p\in M\setminus \overline{M}'_\delta$, then $B^g_R(p)\subset
M\setminus \overline{M}_\delta$. Then the smooth maps $\rho\colon
M\setminus \overline{M}'_\delta\to\cal{M}$ and $\sigma\colon
M\setminus
\overline{M}'_\delta\to\C^{1,\alpha}(B\times(-N,N))$
defined by
\[
\rho(p):= \varphi_p^*g\,,\qquad \si(p):= f(\varphi_p(\cdot),\cdot)
\]
satisfy that 
$$
\|\rho(p)-g_\infty\|_{\C^{2,\alpha}(B)}<\delta\,,\qquad
\|\sigma(p)-G_\infty\|_{\C^{1,\alpha}(B\times(-N,N))}<\delta\,.
$$
Therefore, for each $(p,\epsilon,\bar{v})$ with $p\in M\setminus
\overline{M}'_\delta$, $\epsilon<\delta$ and
$\|\bar{v}\|_{\C^{2,\alpha}(S^{n-1})}<\delta$, it is easy to see that
one can take $(v_0,\hu)$ depending smoothly on $p$, $\epsilon$ and
$\bar{v}$: just define 
$$
v_0(p,\epsilon,\bar{v}):=v_0(\rho(p),\sigma(p),\epsilon,\bar{v})\,,\qquad
\hu(p,\epsilon,\bar{v})=\hu(\rho(p),\sigma(p),\epsilon,\bar{v})\,.
$$
Combining this with Proposition~\ref{prop:implicit} applied to
$\overline{M}'_\delta$, and using the hypothesis on the smallness of $\|\phi_p-\phi_\infty\|_{\C^{2,\alpha}(B_1)}$ in
the first item of Assumption~B to derive the smooth dependence of $(v_0,\hu)$ on $p\in M$, one concludes the proof.
\end{proof}

\begin{proposition}\label{prop:affine_asymp}
	The claim in Proposition~\ref{prop:affine} also holds in the case that $M$ is an asymptotically homogeneous Riemannian manifold.
\end{proposition}
\begin{proof}
	We consider the operator $\cal{F}$ introduced in~\eqref{eq:F} but, instead of being defined in an open set of $M\times[0,\infty)\times\C^{2,\alpha}\avg(S^{n-1})$, we define it on the open set $\cal{V}$ introduced in the proof of Proposition~\ref{prop:implicit_asymp} above. We also adapt the definition of $A$, which now should be defined in $\cal{V}$. In particular, we have $(\exp^h_{p_\infty})^{-1}(c(h,s,\epsilon,\bar{v}))=\epsilon A(h,s,\epsilon,\bar{v})$, where $c(h,s,\epsilon,\bar{v})$ is the center of mass of $\partial B_{\epsilon(1+\bar{v}+v_0)}^h({p_\infty})$. 
	
	Thus, the proof of Proposition~\ref{prop:affine} carries over to the new setting. We define the smooth map
	\[
	\bar{\cal{F}}(h,s,\epsilon,\bar{v}):=(A(h,s,\epsilon,\bar{v}),Q \cal{F}(h,s,\epsilon,\bar{v})),
	\]
	and we apply the implicit function theorem to $\bar{\cal{F}}$
        at the point $(g_\infty,G_\infty,0,0)$, using
        Propositions~\ref{prop:H} and~\ref{prop:linearizationF}. This
        implies the existence of a new $\delta>0$ such that, for each
        metric $h$ with $\|h-g_\infty\|_{\C^{2,\alpha}(B)}<\delta$,
        each function $s$ with
        $\|s-G_\infty\|_{\C^{1,\alpha}(B\times(-N,N))}<\delta$, and each
        $\epsilon<\delta$, there exist
        $\bar{v}\in\C^{2,\alpha}\avg(S^{n-1})$ and $a\in T_{p_\infty} B$,
        smoothly depending on $h$, $s$ and $\epsilon$, such that
        $A(h,s,\epsilon,\bar{v})=0$ and 
$$
\cal{F}(h,s,\epsilon,\bar{v})+\langle a,\cdot\rangle=0.
$$
	
	For such $\delta$, let $M_\delta$ be a relatively compact
        domain in $M$ that satisfies simultaneously the condition
        stated in Definition~\ref{def:asymp} and item~(i) in
        Assumption~B, and $M'_\delta\supset M_\delta$ another
        relatively compact domain such that if $p\in M\setminus
        \overline{M}'_\delta$, then $B^g_R(p)\subset M\setminus
        \overline{M}_\delta$. Again, the smooth maps $\rho$ and
        $\sigma$ introduced in the proof of
        Proposition~\ref{prop:implicit_asymp} satisfy that
        $$
\|\rho(p)-g_\infty\|_{\C^{2,\alpha}(B^{g_\infty}_R(p_\infty))}<\delta\,,\qquad
\|\sigma(p)-G_\infty\|_{\C^{1,\alpha}(B\times(-N,N))}<\delta\,.
$$
        Therefore, for each $p\in M\setminus
        \overline{M}'_\delta$ and $\epsilon<\delta$, we can define 
\[
\bar{v}_{\epsilon,p}:=\bar{v}(\rho(p),\sigma(p),\epsilon)\,,\qquad
a_{\epsilon,p}:=a(\rho(p),\sigma(p),\epsilon)\,,
\]
which depend smoothly on $p$ and $\epsilon$, and satisfy that 
\[
\cal{F}(p,\epsilon,\bar{v}_{\epsilon,p})+\langle
a_{\epsilon,p},\cdot\rangle=0
\]
and the center of mass of $\partial B^g_{\epsilon(1+v)}(p)$ is~$p$. Combining this with Proposition~\ref{prop:implicit} applied to $\overline{M}'_\delta$, one concludes the proof.
\end{proof}

In order to conclude the adaptation of the proof of Theorem~\ref{th:main2} to the asymptotically homogeneous setting, one just has to deal with the proof of Proposition~\ref{prop:critical_point}. Similarly as above, the idea is to analyze separately the situation in a sufficiently large $\overline{M}_\delta$ and in its complement $M\setminus\overline{M}_\delta$. The proof of Proposition~\ref{prop:critical_point} applies directly to the former, and yields the existence of a suitable $\epsilon_{0,\overline{M}_\delta}$. To deal with the latter, one just has to notice that the norms of the objects $a_{\epsilon,p}$ and $b_{\epsilon, p}$ in the proof of Proposition~\ref{prop:critical_point} differ from constants $a_\epsilon$ and $b_\epsilon$ as little as desired, by taking $\delta>0$ small enough. This yields the existence of $\epsilon_{0,M\setminus \overline{M}_\delta}>0$ in the conditions of the statement of the mentioned proposition, whenever $p\in M\setminus \overline{M}_\delta$. Taking $\epsilon_0$ as the minimum of $\epsilon_{0,\overline{M}_\delta}$ and $\epsilon_{0,M\setminus \overline{M}_\delta}$, we are done. 

Finally, we argue how the asymptotically homogeneous assumption guarantees the existence of solutions to the overdetermined
problem~\eqref{eq:onp}. Indeed, by construction, the function
$\cal{J}_\epsilon$, for each $\epsilon<\epsilon_0$, is smooth on
$M$. Also by construction, $\cal{J}_\epsilon$ is asymptotically
constant, that is, there exists $c_\epsilon\in\R$ such that for each
$\delta>0$ there is a relatively compact domain $M_\delta$ in $M$
satisfying $|\cal{J}_\epsilon(p)-c_\epsilon|<\delta$, for all $p\in
M\setminus M_\delta$. Then, a standard argument guarantees the
existence of an extreme value for $\cal{J}_\epsilon$ in~$M$. Hence,
$\cal{J}_\epsilon$ has a critical point, which implies the existence
of solutions to~\eqref{eq:onp}.

Hence we infer that if Assumption~B holds, there exist $\epsilon_0>0$
and a smooth function $\cal{J}_\epsilon\colon M\to\R$ for each
$\epsilon\in(0,\epsilon_0)$, such that, if $p$ is a critical point of
$\cal{J}_\epsilon$, then there exist a $\C^{2,\alpha}$-domain
$\Omega_{\epsilon,p}$ containing $p$, with
$\vol_g(\Omega_{\epsilon,p})=|B_\epsilon|$,  and a function
$u_{\epsilon,p}\in\C^{2,\alpha}(\bar{\Omega}_{\epsilon,p})$, such that
$u_{\epsilon,p}$ is a positive solution to the overdetermined
problem~\eqref{eq:onp} on the domain $\Omega_{\epsilon,p}$ with
$\la:=\ep^{-2}$. Here, $\Omega_{\epsilon,p}$ is a
$\C^{2,\alpha}$-small perturbation of the geodesic ball
$B_\epsilon^g(p)$. Moreover, each function $\cal{J}_\epsilon$ has at
least a critical point and, hence, the existence of a family of
solutions $(\Omega_{\epsilon,p},u_{\epsilon,p},\epsilon^{-2})$ to the
overdetermined problem~\eqref{eq:onp} for all
positive $\epsilon<\epsilon_0$ is assured.
\end{proof}

\section{Symmetry results on manifolds of nonconstant curvature}\label{sec:symmetry}

In this section we investigate overdetermined boundary
problems on spaces with a high
degree of symmetry (that is, with a large isometry group), but maybe not
high enough to have constant curvature. We will show that the
solution domains we have constructed to certain overdetermined
problems inherit some symmetries from the ambient space and we will exploit
this fact to prove partial symmetry results. To avoid
complicated or unnatural assumptions, in this section we restrict our attention to the
subclass of overdetermined problems of form~\eqref{eq:onp_laplacian},
that is, with position-independent nonlinearities.

We start with a purely technical result that we will subsequently
employ to derive more visual consequences. To state it in an economic
way, we will sometimes write the solutions to the overdetermined
boundary problem~\eqref{eq:onp_laplacian} as a triple $(u,\Om,\la)$
and borrow some notation used in Step~1 of the proof of Theorem~\ref{th:main2}:

\begin{proposition}\label{prop:symmetry}
	Let $(M,g)$ be a compact or homogeneous Riemannian manifold
        and assume that the nonlinearity $G(z)$ is
        position-independent and satisfies Assumption~A. Then:
	\begin{enumerate}
		\item Suppose that $(B^g_{\epsilon(1+v_p)}(p),u_p,\epsilon^{-2})$ and
                  $(B^g_{\epsilon(1+v_q)}(q),u_q,\epsilon^{-2})$ are
                  solutions to the overdetermined boundary
                  problem~\eqref{eq:onp_laplacian}, with the domains
                  $B^g_{\epsilon(1+v_p)}(p)$ and
                  $B^g_{\epsilon(1+v_q)}(q)$ being centered at points~$p$ and
                  $q$ of~$M$, respectively, and of the same volume: $\vol_g
                  \bigl(B^g_{\epsilon(1+v_p)}(p)\bigr)=\vol_g
                  \bigl(B^g_{\epsilon(1+v_q)}(q)\bigr)$. There is some
                  $\de>0$ such that, if
                  $\varphi$ is a local isometry of
                  $M$ with $\varphi(p)=q$ and
\[
\epsilon+\|v_p\|_{\C^{2,\alpha}(S^{n-1})}+\|v_q\|_{\C^{2,\alpha}(S^{n-1})}
+ \|\hu_p-\phi\|_{\C^{2,\alpha}(B_1)}
+\|\hu_q-\phi\|_{\C^{2,\alpha}(B_1)} <\delta,
\]
then $\varphi(B^g_{\epsilon(1+v_p)}(p))=B^g_{\epsilon(1+v_q)}(q)$ and $u_p=u_q\circ\varphi$.
		
		\item There exists $\delta>0$ such that, if $\varphi$
                  is an isometry of $M$ with $\varphi(p)=p$ for some
                  $p\in M$, and
                  $(B^g_{\epsilon(1+v)}(p),u,\epsilon^{-2})$ is a
                  solution to \eqref{eq:onp_laplacian}, where the domain
                  $B^g_{\epsilon(1+v)}(p)$ is centered at $p$ and
\[
\epsilon+\|v\|_{\C^{2,\alpha}(S^{n-1})} +\|\hu-\phi\|_{\C^{2,\alpha}(B_1)} <\delta,
\]
then $\varphi(B^g_{\epsilon(1+v)}(p))=B^g_{\epsilon(1+v)}(p)$ and $u=u\circ\varphi$.
	\end{enumerate}
\end{proposition}
\begin{proof}
	In view of Proposition~\ref{th:quotient}, Proposition~\ref{cor:unique} holds both in the compact and the homogeneous settings. Thus, let $\delta$ be as in Proposition~\ref{cor:unique}. 
	
	By the equivariance of the Laplace-Beltrami operator with
        respect to isometries, we have that
        $(\varphi(B^g_{\epsilon(1+v_p)}(p)),u_{p}\circ\varphi^{-1},\epsilon^{-2})$
        is also a solution to \eqref{eq:onp_laplacian}. Notice that
        $\varphi^{-1}(B^g_{\epsilon(1+v_q)}(q))=B^g_{\epsilon(1+w)}(p)$,
        where $w=v_q\circ\varphi_{*p}$. Moreover, $\varphi\circ
        Y_{\epsilon,w,p}=Y_{\epsilon,v_q,q}\circ\varphi_{*p}$, where
        $Y_{\epsilon,v,p}$ refers to the parametrization defined
        in~\eqref{eq:Y}. Thus, we have
        $\widehat{u_q\circ\varphi}=u_q\circ\varphi\circ
        Y_{\epsilon,w,p}=u_q\circ Y_{\epsilon,v_q,q}\circ
        \varphi_{*p}=\hu_q\circ\varphi_{*p}$. This, together with the
        facts that $\phi$ is radially symmetric and $\varphi_{*p}$ is
        an orthogonal transformation, implies that
\[
\|\widehat{u_q\circ\varphi}-\phi\|_{\C^{2,\alpha}(B_1)} =
\|\hu_q-\phi\|_{\C^{2,\alpha}(B_1)}<\delta\,.
\]
Since moreover
\[
\|w\|_{\C^{2,\alpha}(S^{n-1})}=\| v_{q}\circ\varphi_{*p}
\|_{\C^{2,\alpha}(S^{n-1})} = \|v_q\|_{\C^{2,\alpha}(S^{n-1})}<\delta\,,
\]
Proposition~\ref{cor:unique} applies, so it follows that 
\[
\varphi^{-1}(B^g_{\epsilon(1+v_q))}(q))=B^g_{\epsilon(1+v_p)}(p), \qquad \text{and}\qquad
u_{q}\circ\varphi =u_p\,,
\]
which proves~(i).
		
	Part (ii) follows in the same way from the previous arguments.
\end{proof}

Part (i) of Proposition~\ref{prop:symmetry} implies in particular
that, for a compact or homogeneous Riemannian manifold $M$ with
isometry group $\G:=\mathrm{Isom}(M)$, the collection of solutions
to~\eqref{eq:onp_laplacian} constructed in the proof of
Theorem~\ref{th:main2} from the same solution $(B_1,\phi,1)$ to
\eqref{eq:onp_laplacian} are invariant under~$\G$. Part~(ii) means
that, if 
\[
\G_p:=\{h\in \G:h(p)=p\}
\]
is the isotropy group at some point $p\in M$, then the solutions to~\eqref{eq:onp_laplacian} constructed in the proof of Theorem~\ref{th:main2} around the point $p$ are invariant under $\G_p$. In other words, if a small perturbed ball centered at $p$ is a solution domain to~\eqref{eq:onp_laplacian} constructed as in the proof of Theorem~\ref{th:main2}, then such perturbed ball is invariant under $\G_p$. 

\begin{remark}\label{rem:unique_phi}
	If the Dirichlet problem 
\[
\Delta_{\hat{g}} u+ G(u)=0 \text{ in } B_1,\qquad u\rvert_{\partial B_1}=0
\]
admits a unique positive solution for each metric $\hat{g}$ on $B_1$ that is close enough to the Euclidean one, then in Proposition~\ref{prop:symmetry} (and in the remaining results of this section) we can remove the requirement that the solutions involved are close enough to $\phi$. In this case, not only the solution domains to~\eqref{eq:onp_laplacian} constructed as in the proof of Theorem~\ref{th:main2} are invariant under the isotropy group $\G_p$, but also any solution domain that is a small perturbed ball centered at $p$. 
Indeed, if $(B^g_{\epsilon(1+v)}(p),u,\epsilon^{-2})$ is a solution to~\eqref{eq:onp_laplacian} and $\epsilon$ and $\|v\|_{C^{2,\al}(S^{n-1})}$ are small enough, then Proposition~\ref{prop:implicit} guarantees the existence of a solution $\hat{u}'$ to the Dirichlet problem for $\Delta_{\hat{g}}\hat{u}'+G(\hat{u}')=0$ in $B_1$ (where $\hat{g}$ is the pullback metric of $g\rvert_{B^g_{\epsilon(1+v)}(p)}$ under the parametrization $Y$ given in~\eqref{eq:Y}). By the uniqueness assumption in the beginning of this remark we get that $\hat{u}'=\hat{u}$ and, as $\|\hat{u}'-\phi\|_{C^{2,\al}(B_1)}$ is small by Proposition~\ref{prop:implicit}, then also $\|\hat{u}-\phi\|_{C^{2,\al}(B_1)}$ is small, which proves our claim.

The uniqueness for such Dirichlet problem holds, for example, for all concave functions $G$ such that $G(0)>0$, and either $G>0$ and $\lim_{t\to +\infty}G(t)t^{-1}=0$, or that there exists $\beta>0$ such that $G(\beta)=0$, see~\cite[\S2.2]{Lions}. This applies, for instance, to guarantee that small perturbations of small geodesic balls that are solution domains to the overdetermined linear problem $\Delta u+1=0$ are invariant under the isotropy. Another class of nonlinearities satisfying uniqueness of solutions is that of those nonnegative, strictly increasing and strictly convex functions $G$ with $G(0)=0$ and such that $\lim_{z\to\infty} G'(z)\leq (\lambda_2(B_1) /\lambda_1(B_1)-\delta)G'(0)$ for some positive $\delta<\lambda_2(B_1) /\lambda_1(B_1)$, see~\cite{Amann:mathz} (the introduction of $\delta$ with respect to the assumptions in~\cite{Amann:mathz} arises from the need of the uniqueness result not only for the Euclidean metric, but for all sufficiently close-by metrics on $B_1$).
\end{remark}


Now we derive some interesting consequences of
Proposition~\ref{prop:symmetry}. The first one applies to the
so-called two-point homogeneous spaces, which are those Riemannian
manifolds $M$ such that, for any two pairs of points $p_1$, $p_2$,
$q_1$, $q_2\in M$, there exists an isometry of $M$ mapping $p_1$ to
$q_1$ and $p_2$ to $q_2$. It is well-known that a Riemannian manifold
is two-point homogeneous if and only if it is homogeneous and
isotropic, that is, for any tangent vectors $v\in T_p M$ and $w\in T_q
M$, there exists an isometry $\varphi$ of $M$ such that $\varphi(p)=q$
and $\varphi_{*p}v=w$; or equivalently, for each $p\in M$ the isotropy
group $\G_p$ acts transitively on the unit sphere of $T_pM$, via the
isotropy representation $\G_p\times T_p M\to T_p M$ defined as 
$$
(\varphi,v)\mapsto \varphi_{*p}v.
$$
Moreover, these manifolds are Riemannian symmetric spaces~\cite{Szabo}, that is, the geodesic symmetry around each point is a global isometry~\cite{Helgason}. Therefore, the two-point homogeneous spaces are precisely the Euclidean spaces $\R^n$, their symmetric quotients, and the rank one symmetric spaces, that is, the round spheres $S^n$, the real projective spaces $\R P^n$, the complex projective spaces $\mathbb{C} P^n$, the quaternionic projective spaces $\mathbb{H} P^n$, the Cayley projective plane $\mathbb{O}P^2$, the real hyperbolic spaces $\R H^n$, the complex hyperbolic spaces $\mathbb{C} H^n$, the quaternionic hyperbolic spaces $\mathbb{H} H^n$, and the Cayley hyperbolic plane $\mathbb{O} H^2$. For more information on these notions and classifications, see~\cite[Appendices A.3-4]{BCO2} and~\cite{Helgason}. 

The following result asserts that, in a two-point homogeneous space
with a position-independent nonlinearity, the only small perturbations
of a small geodesic sphere where the overdetermined problem admits a
nontrivial solution are precisely geodesic spheres, and the solution
is radially symmetric (i.e.\ it depends only on the distance to the
center). The result is a direct consequence of
Propositions~\ref{cor:unique} and~\ref{th:quotient}, and of the fact that, in a two-point homogeneous space, $\G_p$ acts transitively on the unit sphere of $T_pM$, for all $p\in M$.
\begin{corollary}\label{cor:2point}
	Let $M$ be a two-point homogeneous Riemannian manifold and
        assume that the nonlinearity $G(z)$ satisfies
        Assumption~A. Then there exists $\delta>0$ such that, if
        $(B^g_{\epsilon(1+v)}(p),u,\epsilon^{-2})$ is a solution to
        \eqref{eq:onp_laplacian}, where $B^g_{\epsilon(1+v)}(p)$ is
        centered at $p$, with 
$$
\epsilon+\|v\|_{\C^{2,\alpha}(S^{n-1})}+\|\hu-\phi\|_{\C^{2,\alpha}(B_1)}
<\delta,
$$
then $B^g_{\epsilon(1+v)}(p)$ is a geodesic ball and $u$ is radially symmetric around~$p$.
\end{corollary}
	
Another interesting case occurs when the ambient manifold is a
Riemannian symmetric space $M:=\G/\K$ (not necessarily of rank one),
where $\G$ is the identity connected component of the isometry group
of $M$ and $\K:=\G_{p_0}$ is the isotropy group at some base point
$p_0\in M$. Here $\G$ acts transitively on $M$, which implies that $M$
is homogeneous. Let $k$ be the rank of $M$, that is, the dimension of
the smallest totally geodesic and flat submanifold of $M$ (which is
called a maximal flat of $M$). It is a well-known fact~\cite{HPTT,
  Ko:tams} that the isotropy group $\G_p$ at any point $p$ of
$M=\G/\K$ acts isometrically with cohomogeneity $k$ on $M$, that is,
the orbits of maximal dimension of the $\G_p$-action on $M$ have
codimension $k$ in $M$. Moreover, this action is hyperpolar, which
means that there exists a totally geodesic and flat submanifold
$\Sigma$ of $M$ (which turns out to be a maximal flat) that intersects
all the $\G_p$-orbits and always perpendicularly. In this case, the
Weyl group of the action~\cite{Ko:jdg} is the discrete group of
reflections of $\Sigma$ given by
\[
W_\Sigma:=N_{\G_p}(\Sigma)/Z_{\G_p}(\Sigma),
\]
where
\[
N_{\G_p}(\Sigma):=\{g\in \G_p:g(\Sigma)=\Sigma\},\quad
Z_{\G_p}(\Sigma):=\{g\in \G_p:g(q)=q \text{ for all }q\in \Sigma\}.
\]
Then, it is immediate to derive the following application of Propositions~\ref{cor:unique} and~\ref{th:quotient}.
	
\begin{corollary}
Let $M$ be a symmetric space of dimension $n\geq 2$ and rank $k$, and
let the nonlinearity $G(z)$ satisfy Assumption~A. Then there exists
$\delta>0$ such that, if $(B^g_{\epsilon(1+v)}(p),u,\epsilon^{-2})$ is
a solution to \eqref{eq:onp_laplacian}, where $B^g_{\epsilon(1+v)}(p)$
is centered at $p$ and 
\[
\epsilon +\|v\|_{\C^{2,\alpha}(S^{n-1})}
+\|\hu-\phi\|_{\C^{2,\alpha}(B_1)} <\delta,
\]
then the solution $(B^g_{\epsilon(1+v)}(p),u,\epsilon^{-2})$ is invariant under a group of isometries of $M$ that acts with cohomogeneity $k-1$ on the hypersurface $\partial B^g_{\epsilon(1+v)}(p)$.

Moreover, let $\Sigma$ be a maximal flat through $p$. Then there exists a hypersurface $L=\partial B^g_{\epsilon(1+v)}(p)\cap \Sigma$ in $\Sigma$ which is invariant under the Weyl group $W_\Sigma$, and such that $\partial B^g_{\epsilon(1+v)}(p)$ is the union of all the $\G_p$-orbits through the points of $L$.
\end{corollary}

An interesting problem is to investigate the shape of the hypersurface $L$ of $\Sigma$, in particular, if $M$ has rank two and, hence, $L$ is a Jordan curve in $\Sigma$. A similar problem, related to the shape of soap bubbles in symmetric spaces, has been studied in~\cite{Hsiang}.

Corollary~\ref{cor:2point} admits the following interesting improvement on harmonic
spaces. Recall that a Riemannian manifold $M$ is said to be locally harmonic at
$p\in M$ if all geodesic spheres of sufficiently small radius centered
at $p$ have constant mean curvature. $M$ is a harmonic space if it is
locally harmonic at all points. These spaces have not been classified
yet, and the only known examples are the two-point homogeneous spaces
mentioned above, their quotients, and certain solvable Lie groups
endowed with a left-invariant metric called Damek-Ricci spaces; for
more information on harmonic spaces, see~\cite{BTV}
and~\cite[Chapter~6]{Besse}. Notice that this result does not require
the manifold to be homogeneous or compact:

\begin{theorem}\label{th:harmonic}
	Let $M$ be a Riemannian manifold that is locally harmonic at $p\in M$ and assume that the
        nonlinearity $G(z)$ satisfies Assumption~A. Then there exists
        $\delta>0$ such that, if
        $(B^g_{\epsilon(1+v)}(p),u,\epsilon^{-2})$ is a solution to
        \eqref{eq:onp_laplacian}, where $B^g_{\epsilon(1+v)}(p)$ is
        centered at $p$ and
\[
\epsilon +\|v\|_{\C^{2,\alpha}(S^{n-1})}
+\|\hu-\phi\|_{\C^{2,\alpha}(B_1)} <\delta,
\]
then $B^g_{\epsilon(1+v)}(p)$ is a geodesic ball and $u$ is radially symmetric around~$p$.
\end{theorem}

\begin{proof}
	Let $\C^{2,\alpha}_{0,0}([0,1])$ denote the Banach space of functions in $\C^{2,\alpha}([0,1])$ that vanish at $1$ and whose first derivatives vanish at $0$. Let $\epsilon_0>0$ be such that $\epsilon_0(1+\epsilon_0)$ is less than the injectivity radius of $M$. Consider the map
	\[
	\begin{array}{rccl}
	\cal{N}_p\colon & [0,\epsilon_0)\times(-\epsilon_0,\epsilon_0)\times \C^{2,\alpha}_{0,0}([0,1]) & \to & \C^{0,\alpha}([0,1])
	\\
	&(\epsilon,\eta,\psi)& \mapsto & \psi''+(\frac{n-1}{r}+\frac{\widehat{\theta}_p'}{\widehat{\theta}_p})\psi'+G\circ\psi.
	\end{array}
	\]
	Here, $\widehat{\theta}_p=\theta_p\circ Y$, where $Y=\exp_p^{g}\,\circ\, T_{\epsilon(1+\eta)}\colon B_1\subset T_p M\to B_{\epsilon(1+\eta)}^{g}(p)$, and $\theta_p$ is the function that measures the infinitesimal change of volume around $p$, which is radial since $M$ is locally harmonic at $p$ (see~\cite[Chapter~6, \S{}A-B]{Besse}), and hence, $\widehat{\theta}_p$ (as well as $\psi$) is a function of the radial distance $r$. Notice that we can regard $\cal{N}_p$ as the operator $\Delta_{\hg}+G(\cdot)$ applied to a radial function $\psi\in\C^{2,\alpha}_{0,0}(B_1)$, where $\hg=\epsilon^{-2}Y^*g$.
	
	Since the function $\phi$ in Assumption~A is radial, abusing
        the notation again, we denote also by $\phi$ the corresponding
        function in $\C^{2,\alpha}_{0,0}([0,1])$. Then
        $\cal{N}_p(0,0,\phi)=0$ and
        $D_\psi\cal{N}_p(0,0,\phi)=(\cdot)''+\frac{n-1}{r}(\cdot)'+G'\circ\phi$. Since
        by Assumption~A the operator $\Delta+G'\circ\phi$ is
        invertible, the operator
\[
(\cdot)''+\frac{n-1}{r}(\cdot)'+G'\circ\phi\colon
\C^{2,\alpha}_{0,0}([0,1])\to\C^{0,\alpha}([0,1])
\]
is also invertible.
	Then $D_{\psi}\cal{N}_p(0,0,\phi)$ is invertible and, hence, the implicit function theorem guarantees, for each $\epsilon$ and $\eta$ small enough, the existence of $\psi$ such that $\cal{N}_p(\epsilon,\eta,\psi)=0$. Therefore, there exists a radial solution $(B^g_{\epsilon(1+\eta)}(p),\psi\circ Y^{-1},\epsilon^{-2})$ to~\eqref{eq:onp_laplacian}; notice that the Neumann condition is satisfied trivially, since $\psi\circ Y^{-1}$ is radial. 
	
	Finally, let $\eta\in\R$ be such that $\vol(B^g_{\epsilon(1+\eta)}(p))=\vol_g(B^g_{\epsilon(1+v)}(p))$. Then, the application of Proposition~\ref{cor:unique} to the triples $(B^g_{\epsilon(1+v)}(p),u,\epsilon^{-2})$  and $(B^g_{\epsilon(1+\eta)}(p),\psi\circ Y^{-1},\epsilon^{-2})$ concludes the proof.
\end{proof}

Observe that Theorem~\ref{th:harmonic} combined with Remark~\ref{rem:unique_phi} implies Theorem~\ref{th:harmonic_intro} in the Introduction.

\begin{remark}\label{R.la1}
	Although the arguments developed in Section~\ref{sec:setup} do not apply directly to the first
        eigenvalue problem $\Delta_g u+\lambda_1 u=0$ because $G(t):=t$
        does not satisfy Assumption~A, the
        symmetry results in this section apply to this case as
        well. The reason is that, by construction, the uniqueness
        result in Proposition~\ref{cor:unique} holds for the solutions constructed in~\cite{DS:DCDS,PS:AIF} to the first eigenvalue problem, just by imposing the natural requirement that the solutions are normalized to be positive and have $L^2$-norm equal to $1$. Since, under this additional constraint, we have uniqueness of solutions to the Dirichlet problem, Remark~\ref{rem:unique_phi} applies and, hence, no requirements on the proximity of $\hu$ to $\phi$ are needed.  
\end{remark}

\section*{Acknowledgements}

The first author is supported by projects MTM2016-75897-P (AEI/FEDER) and ED431F 2017/03, and by the European Union's Horizon 2020 research and innovation programme under the Marie Sklodowska-Curie grant agreement No.~745722. The second and third authors
have been supported by the ERC Starting Grants 633152 and 335079,
respectively. All authors acknowledge financial support from the
Spanish Ministry of Economy and Competitiveness, through the Severo
Ochoa Program for Centers of Excellence in R\&D (SEV-2015-0554).

\end{document}